\documentclass[letterpaper,review,preprint,12pt]{elsarticle}
\pdfoutput=1
\journal{}

\usepackage{amssymb}
\usepackage{amsthm}
\usepackage{amsmath}
\usepackage{upgreek}
\usepackage[eulergreek]{sansmath}
\usepackage{pdflscape}
\usepackage{multirow}
\usepackage{multicol}
\usepackage{color}
\usepackage{stmaryrd}
\usepackage{listings}
\usepackage{varwidth}
\usepackage{booktabs}
\usepackage[section]{placeins} 
\usepackage{calc}
\usepackage[section]{algorithm}
\usepackage[titletoc]{appendix}

\theoremstyle{remark}
\newtheorem*{remark}{Remark}

\usepackage{float}
\usepackage{wrapfig}
\usepackage{epstopdf}
\usepackage[percent]{overpic}
\usepackage{graphicx}
\usepackage{graphics}
\usepackage{graphicx}
\usepackage{subcaption}
\usepackage{varwidth}

\usepackage{tikz}
\usepackage{pgfplots}
\usetikzlibrary{arrows,matrix,positioning,fit}
\usetikzlibrary{shapes,positioning}
\usetikzlibrary{backgrounds}
\usepgfplotslibrary{colorbrewer}
\usepgfplotslibrary{patchplots}
\usepgfplotslibrary[colorbrewer]
\usetikzlibrary{pgfplots.colorbrewer}
\usetikzlibrary[pgfplots.colorbrewer]
\usetikzlibrary{pgfplots.groupplots}
\usepgfplotslibrary{units}
\usetikzlibrary{spy}
\usepackage{pgfplotstable}
\graphicspath{ {./Figs/} }

\newsavebox\mybox

\newlength\myheight
\newlength\mydepth
\settototalheight\myheight{Xygp}
\newcommand*\inlinegraphics[1]{%
  \settototalheight\myheight{Xygp}%
  \settodepth\mydepth{Xygp}%
  \raisebox{-\mydepth}{\includegraphics[height=\myheight]{#1}}%
}
\newcommand\orcid[1]{\href{https://orcid.org/#1}{\inlinegraphics{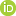}}}

\makeatletter
\def\BState{\State\hskip-\ALG@thistlm}
\makeatother

\newtheorem{theorem}{Theorem}[section]

\newtheorem{lemma}[theorem]{Lemma}

\newcommand{\etal}{et~al.}

\newcommand{\hfd}{\hat{f}^{\delta}}

\newcommand{\hfI}{\hat{f}^{\delta I}}
\newcommand{\hfC}{\hat{f}^{\delta C}}
\newcommand{\hfD}{\hat{f}^{\delta D}}

\newcommand{\hud}{\hat{u}^{\delta}}

\newcommand{\bud}{\mathbf{u}^{\delta}}

\newcommand{\huI}{\hat{u}^{\delta I}}

\newcommand\px[2]{\frac{\partial #1}{\partial {#2}}}
\newcommand\pxi[3]{\frac{\partial^{#1}#2}{\partial {#3}^{#1}}}
\newcommand\dx[2]{\frac{\mathrm{d} #1}{\mathrm{d} #2}}
\newcommand\dxi[3]{\frac{\mathrm{d}^{#1}#2}{\mathrm{d} {#3}^{#1}}}
\newcommand\rint[2]{\int^{1}_{-1}{#1} \mathrm{d}{#2}}
\newcommand\poch[2]{({#1})_{#2}}
\newcommand{\half}{\frac{1}{2}}

\newcommand\CONDITION[2]%

\begin{document}

\begin{frontmatter}

\title{A New Family of Weighted One-Parameter Flux Reconstruction Schemes}

\author{W. Trojak\corref{cor1}~\orcid{0000-0002-4407-8956}}
\ead{wt247@tamu.edu}
\cortext[cor1]{Corresponding author}
\address{Department of Ocean Engineering, Texas A\&M University, College Station}

\author{F. D. Witherden\corref{cor2}~\orcid{0000-0003-2343-412X}}
\ead{fdw@tamu.edu}
\address{Department of Ocean Engineering, Texas A\&M University, College Station}

\begin{abstract}
    The flux reconstruction (FR) approach offers a flexible framework for describing a range of high-order numerical schemes; including nodal discontinuous Galerkin and spectral difference schemes. 
    This is accomplished through the use of so-called correction functions.
    In this study we employ a weighted Sobolev norm to define a new extended family of FR correction functions, the stability of which is affirmed through Fourier analysis. 
    Several of the schemes within this family are found to exhibit reduced dissipation and dispersion overshoot.
    Moreover, many of the new schemes possess higher CFL limits whilst maintaining the expected rate of convergence.
    Numerical experiments with homogeneous linear convection and Burgers’ turbulence are undertaken, and the results observed to be in agreement with the theoretical findings.
\end{abstract}


\end{frontmatter}



\section{Introduction}
\label{sec:intro}

Discontinuous spectral element methods have emerged as an attractive alternative to classical finite element and finite volume methods for high-order accurate numerical simulations on unstructured grids.  Such methods offer the promise of increased accuracy at reduced cost \cite{Vincent2011a,Vermeire2017b}.

A popular example of such a method is the Flux Reconstruction (FR) approach of Huynh~\cite{Huynh2007}.  Closely related to the lifting collocation penalty (LCP) schemes of Gao~and~Wang~\cite{Gao2009,Wang2009} and the correction procedure via reconstruction (CPR) method of Haga~\etal~\cite{Haga2011}, the FR approach offers more flexibility as rather than defining \emph{one} scheme it instead defines an \emph{infinite family} of schemes.  This flexibility enables the FR approach to recover several existing spectral element schemes.  These include a nodal form of the discontinuous Galerkin method of Reed~and~Hill~\cite{Reed1973} as described in Hesthaven~and~Warburton~\cite{Hesthaven2007} and, for a linear flux function, the spectral difference (SD) schemes of Kopriva~and~Kolias~\cite{Kopriva1996}, Liu~\etal~\cite{Liu2006}, and Sun~\etal~\cite{Sun2007}.  The flexibility of the FR framework also enables the construction of new and novel schemes.

In 2011, Vincent~\etal~\cite{Vincent2011} identified a one-parameter family of correction functions, herein referred to as original stable FR (OSFR) schemes, that led to stable FR schemes for linear advection problems.  This work was subsequently extended to linear advection-diffusion problems by Castonguay~\etal~\cite{Castonguay2014}.  More recently Vincent~\etal~\cite{Vincent2015} identified a multi-parameter family of linearly stable FR correction functions which are herein referred to as the extended stable FR (ESFR) schemes.  In a series of numerical experiments, Vermeire~and~Vincent~\cite{Vermeire2016a} observed that several of these schemes are more stable for implicit large eddy-simulation (ILES) simulations than NDG.

Further advancements have been made in correction functions with the definition of generalised Sobolev stable FR (GSFR)~\cite{Trojak2018} and generalised Lebesgue stable FR (GLSFR)~\cite{Trojak2018b}. These approaches both vastly increase the scope of FR.  A pictorial illustration of the current space of FR schemes can be seen in Fig.~\ref{fig:venn_diagram1}.  In spite of this progress a broader definition of what correction function is and how one should be chosen far not fully understood.  To this end, in this paper we seek to extend the theoretical understanding of FR correction functions.  We shall accomplish this through the application of a \emph{weighted} Sobolev type norm---something which has previously been employed successfully within the summation-by-parts community.  The use of a weight function also makes it possible to establish additional connections between FR and LCP.  The result of this work is a further generalisation of the OSFR correction functions.  This new family is capable of recovering a range of different SD schemes (as opposed to the single scheme recovered by OSFR), along with several entirely new schemes.
	
	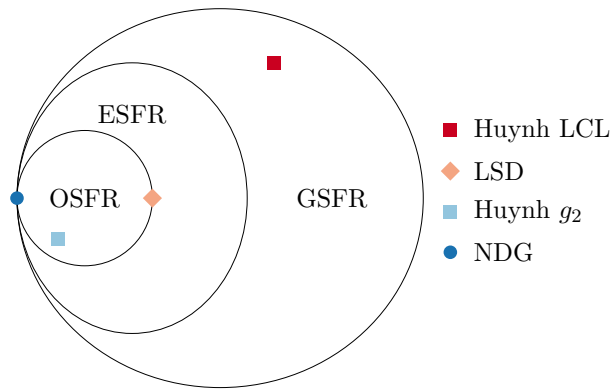
\begin{figure}[tbhp]
		\centering
			\resizebox{0.6\linewidth}{!}{\begin{tikzpicture}[every node/.style={font=\small},scale=2]
	\def\esfr{(0.1,0) ellipse (0.85 and 1)}
	\def\osfr{(-0.25,0) ellipse (0.5 and 0.5)}
	\def\gsfr{(0.75,0) ellipse (1.5 and 1.4)}
	\definecolor{col1}{RGB}{202,0,32}
	\definecolor{col2}{RGB}{244,165,130}
	\definecolor{col3}{RGB}{146,197,222}
	\definecolor{col4}{RGB}{25,113,176}
   
    \begin{scope}[] 
        
        \draw \esfr node[label={[xshift=0, yshift=2em]\small{ESFR}}]{};
        \draw \osfr node[label={[xshift=0em, yshift=-1em]\small{OSFR}}]{};
        \draw \gsfr node[label={[xshift=4em, yshift=-1em]\small{GSFR}}]{};
        
        \fill[fill=col1] (1.1,0.95) rectangle (1.2,1.05);
        \fill[fill=col2,rotate around={45:(0.25,0)}] (0.2,-0.05) rectangle (0.3,0.05);
        \fill[fill=col3] (-0.5,-0.35) rectangle (-0.4,-0.25);
        \fill[fill=col4] (-0.75,0) circle (0.05);
        
        \fill[fill=col1] (2.4,0.45) rectangle (2.5,0.55);
        \node[anchor=west] at (2.55,0.5) {Huynh LCL};
        
        \fill[fill=col2,rotate around={45:(2.45,0.2)}] (2.4,0.15) rectangle (2.5,0.25);
        \node[anchor=west] at (2.55,0.2) {LSD};
        
        \fill[fill=col3] (2.4,-0.15) rectangle (2.5,-0.05);
        \node[anchor=west] at (2.55,-0.1) {Huynh $g_2$};
        
        \fill[fill=col4] (2.45,-0.4) circle (0.05);
        \node[anchor=west] at (2.55,-0.4) {NDG};
    \end{scope}

\end{tikzpicture}}
		\caption{Euler diagram to show the interconnection of the spaces of FR correction functions: Nodal DG~(NDG)~\cite{Huynh2007}; original stable FR~(OSFR)~\cite{Vincent2010}; extended range stable FR~(ESFR)~\cite{Vincent2015} and, generalised Sobolev stable FR~(GSFR)~\cite{Trojak2018}. Some specific examples of specific schemes are given, notably Huynh's lumped Chebyshev--Lobatto~(LCL)~\cite{Huynh2007} scheme and the original Legendre spectral difference (LSD) scheme~\cite{Kopriva1996,Huynh2007,Jameson2010}.}
		\label{fig:venn_diagram1}
	\end{figure}

The remainder of this paper is structured as follows.  In section~\ref{sec:fr} we provide an overview of the FR approach for a one-dimensional advection problem and introduce the one-parameter family of energy stable FR schemes.  Our extension of these schemes is presented in section~\ref{sec:stability} with stability being discussed in section~\ref{sec:limits}.  Here we also detail connections between our new family of schemes and SD schemes.  The explicit time-step limits associated with our new family of schemes are reviewed in section~\ref{sec:conv}.  The performance of these schemes within the context of a heterogeneous linear conservation equation and Burgers' turbulence are assessed in section~\ref{sec:numeric}.  Finally, in section~\ref{sec:conclusions}, conclusions are drawn.
\section{Flux Reconstruction}
\label{sec:fr}

Consider using the FR approach to solve the 1D conservative equation
	\begin{equation}\label{eq:governing}
		\px{u}{t}+\px{f}{x}=0, \quad \mathrm{for} \quad x\in \mathbf{\Omega}, \quad \mathrm{and} \quad t\in \mathbb{R}^+.
	\end{equation}
	The FR approach involves partitioning the domain, $\mathbf{\Omega}$, into $N$ distinct elements, each denoted by $\mathbf{\Omega}_n=\{x|x_n<x<x_{n+1}\}$, such that
	\begin{equation}\label{eq:partition}
		\mathbf{\Omega}=\bigcup_{n=0}^{N-1}\mathbf{\Omega}_n
		\hspace{1cm}\mathrm{and}\hspace{1cm}
		\bigcap_{n=0}^{N-1}\mathbf{\Omega}_n=\emptyset.
	\end{equation}
	The solution $u$ in Eq.~(\ref{eq:governing}) is approximated in each $\mathbf{\Omega}_n$ by $u^{\delta}_n=u^{\delta}_n(x,t)$, which is a polynomial of degree $p$ within $\mathbf{\Omega}_n$, and the flux $f$ is approximated by $f^{\delta}_n=f^{\delta}_n(x,t)$, a degree $p+1$ polynomial within $\mathbf{\Omega}_n$.

	To permit the use of standardised operators, the mapping of each element, $\mathbf{\Omega}_n$, to a reference element, $\hat{\mathbf{\Omega}}=\{\zeta|-1\leq \zeta\leq 1\}$, is calculated
	\begin{equation}\label{eq:mapping}
		\zeta=\Theta_n(x)=2\left(\frac{x-x_n}{x_{n+1}-x_n}\right)-1,
	\end{equation}
	which has the inverse
	\begin{equation}
  		x=\Theta_n^{-1}(\zeta)=\left(\frac{1-\zeta}{2}\right)x_n+\left(\frac{1+\zeta}{2}\right)x_{n+1}.
	\end{equation}
	The governing equation, Eq.~(\ref{eq:governing}), may then be cast into the reference domain as	
	\begin{equation}\label{eq:stand_ele_govern}
		\px{\hud}{t}+\px{\hfd}{\zeta}=0,
	\end{equation}
	where
	\begin{subequations}\label{eq:t-soln}
		\begin{align}
			\hud=\hud(\zeta,t) &= J_{n}u^{\delta}_n(\Theta_n^{-1}(\zeta),t), \\
			\hfd=\hfd(\zeta,t) &= f^{\delta}_n(\Theta_n^{-1}(\zeta),t).
		\end{align}	
	\end{subequations}
	For linearly deformed 1D elements the Jacobian is simply $J_n=(x_{n+1}-x_{n})/2$.

	If inside the reference element, there is a set of points where the solution is known, $\zeta_i$ ($i=\{0\dots p\}$), then $\hud$ and $\hfD$ may be defined with Lagrange polynomials as
	\begin{equation}\label{eq:trans_soln}
		\hud = \sum_{i=0}^{p}\hud_{i}\;l_i \quad \mathrm{and} \quad  \hfD=\sum_{i=0}^{p}\hfD_{i}\;l_i,
	\end{equation}
	where $\hud_i=\hud(\zeta_i)$ and $\hfD_i$ is the flux calculated from $\hud_i$. The Lagrange polynomials in the previous definition take the form
	\begin{equation}
		l_i = \prod_{j=0, j\neq i}^{p}\left(\frac{\zeta-\zeta_j}{\zeta_i-\zeta_j}\right). 
	\end{equation}
	
	The flux polynomial in Eq.~(\ref{eq:trans_soln}) currently represents a transformed flux that is discontinuous between elements. This is made $C^0$ continuous in 1D as 
	\begin{equation}\label{eq:f_total}
		\hfd = \hfD + \hfC = \hfD + (\hfI_L - \hfD_L)h_L + (\hfI_R - \hfD_R)h_R.
	\end{equation}
	$\hfI_L$ is a common flux at the left element interface, $\zeta=-1$. This value is common with a collocated point in the associated adjacent element. $\hfI_R$ is similarly defined.	This common value is typically calculated using an approximate Riemann solver, but any E-flux is sufficient~\cite{Osher1984,Jameson2012}. Here, $h_L$ and $h_R$ are degree $p+1$ correction function polynomials with boundary conditions
	\begin{subequations}\label{eq:h_BCs}
		\begin{align}
			& h_L(-1)=1, \hspace{0.5cm} h_L(1)=0, \\
			& h_R(-1)=0, \hspace{0.5cm} h_R(1)=1.
		\end{align}
	\end{subequations}
	These conditions are such that the corrected flux, $\hfC$, takes the value of $\hfI_L$ and $\hfI_R$ are the respective interfaces. The spatial derivative of the corrected flux may then be taken
	\begin{equation}\label{eq:fr-nodal}
		\px{\hfd}{\zeta}(\zeta_i) = \sum_{j=0}^{p}\hat{f}^{\delta
D}_j\;\dx{l_j}{\zeta}(\zeta_i) + (\hfI_L-\hfD_L)\dx{h_{L}}{\zeta}(\zeta_i) + (\hfI_R-\hfD_R)\dx{h_{R}}{\zeta}(\zeta_i).
	\end{equation}
	As the correction functions were degree $p+1$, the derivative of the corrected flux and the solution are now both in a degree $p$ polynomial space. Hence they may be combined in the governing equation
	\begin{equation}\label{eq:semi_disc}
		\dx{\hud_i}{t}=-\px{\hfd}{\zeta}(\zeta_i).
	\end{equation}
	The solution can then be advanced with a suitable choice of time integration or ODE method.
	
	The nature of a particular FR scheme depends on three factors: the location of the solution points, the interface flux formulation, and the form of the correction functions.  In their 2007 paper introducing FR, Huynh~\cite{Huynh2007} found that a collocation-based nodal DG scheme (NDG) is recovered if the left and right correction functions are taken to be the right and left Radau polynomials, respectively.  They also showed how, at least for a linear flux function, it is possible for FR to recover any SD scheme.  Moreover, Huynh also presented several new and novel schemes with various stability and accuracy properties.

	Through a tensor-product construction the FR approach can be readily extended to handle quadrilateral and hexahedral elements.  Further details can be found in \cite{Huynh2007,Witherden2014}.
	
	The defining feature of the FR methodology is the unification it brings to several schemes through the correction functions.  The identification of such functions is the main concern of the present work, extending the OSFR set of schemes~\cite{Vincent2011}. Specifically, the OSFR corrections functions are obtained by setting:
\[
 h_L = \frac{(-1)^p}{2}\bigg[L_p - \frac{\eta_pL_{p-1} + L_{p+1}}{1 + \eta_p}\bigg],
\]
and
\[
 h_R = \frac{1}{2}\bigg[L_p + \frac{\eta_pL_{p-1} + L_{p+1}}{1 + \eta_p}\bigg],
\]
where $L_p=L_p(\zeta)$ is a Legendre polynomial of degree $p$ and
\[
 \eta_p = \frac{c(2p+1)(a_pp!)^2}{2} \quad \text{and} \quad a_p = \frac{(2p)!}{2^p(p!)^2},
\]
with $c$ being a free scalar parameter which must lie within the range
\[
 \frac{-2}{(2p+1)(a_pp!)^2} < c < \infty.
\]
Such correction functions satisfy
\begin{subequations}\label{eq:osfr_cond}
	\begin{align}
		\int_{-1}^{1} h_L \px{\hud}{\zeta} \,\mathrm{d}\zeta - c\pxi{p}{\hud}{\zeta} \dxi{p+1}{h_L}{\zeta} &= 0, \\
		\int_{-1}^{1} h_R \px{\hud}{\zeta} \,\mathrm{d}\zeta - c\pxi{p}{\hud}{\zeta} \dxi{p+1}{h_R}{\zeta} &= 0.
	\end{align}
\end{subequations}
This property ensures that the resulting scheme will be linearly stable in the broken-Sobolev type norm
\begin{equation}
 \sum^N_{n=1}\|\hud_n\|^2_{W^c_2} = \sum_{n=1}^{N} \int_{x_n}^{x_{n+1}} (u_n^{\delta})^2 + \frac{c}{2}J_n^{2p} \bigg(\pxi{p}{u_n^{\delta}}{x}\bigg)^2 \,\mathrm{d}x.
\end{equation}
i.e. the energy is sufficiently constrained to give a bounded functional space of solutions.

\section{Weighted Energy Stability}
\label{sec:stability}

	Following on from the definition of OSFR, we may define the Jacobi weighted Sobolev norm~\cite{Zhikov1998} in the reference domain as
	\begin{equation}\label{eq:weighted_sobolev}
		\|\hud\|^2_{W_2^{\iota,w}} = \rint{\Bigg(\big(\hud\big)^2 + \iota\bigg(\pxi{p}{\hud}{\zeta}\bigg)^2\Bigg)(1-\zeta)^{\alpha}(1+\zeta)^{\beta}}{\zeta},
	\end{equation}
	noting that we have chosen to use $\iota = c/2$ from comparison with section~\ref{sec:fr}. Here we will add the constraint that $\alpha,\beta<-1$, which is required for the Jacobi polynomial induced later to be well defined~\cite{Ismail2005}. Thus, the criterion on energy stability in time for a finite polynomial solution is that
	\begin{equation}\label{eq:weighted_sobolev_decay}
		\dx{}{t}\|\hud\|^2_{W_2^{\iota,w}} = \dx{}{t}\rint{\Bigg(\big(\hud\big)^2 + \iota\bigg(\pxi{p}{\hud}{\zeta}\bigg)^2\Bigg)w_{\alpha,\beta}(\zeta)}{\zeta} \leqslant 0,
	\end{equation}
	where
	\begin{equation*}
		w_{\alpha,\beta}(\zeta) = (1-\zeta)^{\alpha}(1+\zeta)^{\beta}.
	\end{equation*}
	For brevity we will define the average as
	\begin{equation}\label{eq:wbar}
		 \overline{w}_{\alpha,\beta} = \frac{1}{2}\rint{w_{\alpha,\beta}(\zeta)}{\zeta}.
	\end{equation}
	
	Let us now consider applying FR to a linear advection problem, such that stability can be investigated in this weighted norm.  Without loss of generality we shall assume a unit convection velocity such that $f(u) = u$.  It follows that
\begin{equation}\label{eq:linad_fr}
	\dx{\hud}{t} = - \px{\hud}{\zeta} - (\huI_L-\hud_L)\dx{h_L}{\zeta} - (\huI_R-\hud_R)\dx{h_R}{\zeta}.
\end{equation}
	Previously, for the proof of OSFR, Eq.~(\ref{eq:linad_fr}) would be multiplied by $\hud$ and integrated over the reference domain. However, as we want to use a weight function, we shall defer this integration step, as this simplifies the use of the product rule.  Multiplying Eq.~(\ref{eq:linad_fr}) by $\hud$ we obtain
\begin{equation*}
	\hud\dx{\hud}{t} = - \hud\px{\hud}{\zeta} - (\huI_L-\hud_L)\hud\dx{h_L}{\zeta} - (\huI_R-\hud_R)\hud\dx{h_R}{\zeta},
\end{equation*}
for which the product rule can be used to get
\begin{equation*}
	\begin{split}
		\frac{1}{2}\dx{(\hud)^2}{t} = - \frac{1}{2}\px{(\hud)^2}{\zeta} &- (\huI_L-\hud_L)\bigg(\px{h_L\hud}{\zeta} - h_L\dx{\hud}{\zeta}\bigg)\\	&- (\huI_R-\hud_R)\bigg(\px{h_R\hud}{\zeta} - h_R\dx{\hud}{\zeta}\bigg).
	\end{split}
\end{equation*}
This step is important as it allows for the formation of the conserved variable at the interface as well ensuring that only derivatives of $\hud$ are present.  We may now proceed to multiply by the weight function and integrate over the reference domain as
\begin{equation}\label{eq:u_wsn}
	\begin{split}
		\frac{1}{2}\dx{}{t}\rint{\big(\hud\big)^2w_{\alpha,\beta}}{\zeta} = &- \frac{1}{2}\rint{\px{(\hud)^2}{\zeta}w_{\alpha,\beta}}{\zeta}\\
	 	&- (\huI_L-\hud_L)\rint{\bigg(\px{h_L\hud}{\zeta} - h_L\dx{\hud}{\zeta}\bigg)w_{\alpha,\beta}}{\zeta}\\
	 	&- (\huI_r-\hud_R)\rint{\bigg(\px{h_R\hud}{\zeta} - h_R\dx{\hud}{\zeta}\bigg)w_{\alpha,\beta}}{\zeta}.
	 \end{split}
\end{equation}
Now proceeding to form the second component of the weighted Sobolev norm we first take the $p^{\mathrm{th}}$ spatial derivative of Eq.~(\ref{eq:linad_fr})
\begin{equation}\label{eq:linad_pth}
	\dx{}{t}\pxi{p}{\hud}{\zeta} = - \pxi{p+1}{\hud}{\zeta} - (\huI_L-\hud_L)\dxi{p+1}{h_L}{\zeta} - (\huI_R-\hud_R)\dxi{p+1}{h_R}{\zeta}.
\end{equation}
Given  $\hud$ is a degree $p$ polynomial, Eq.~(\ref{eq:linad_pth}) may be multiplied by the $p^{\mathrm{th}}$ derivative of $\hud$ and integrated over the reference domain with the weighting function to give
\begin{equation}\label{eq:up_wsn}
	\begin{split}
		\frac{1}{2}\dx{}{t}\rint{\Bigg(\pxi{p}{\hud}{\zeta}\Bigg)^2w_{\alpha,\beta}}{\zeta} =
		&- 2(\huI_L-\hud_L)\pxi{p}{\hud}{\zeta}\dxi{p+1}{h_L}{\zeta}\overline{w}_{\alpha,\beta}\\
		&- 2(\huI_R-\hud_R)\pxi{p}{\hud}{\zeta}\dxi{p+1}{h_R}{\zeta}\overline{w}_{\alpha,\beta}.
	\end{split}
\end{equation}
This simplification can be made due to the respective orders of $\hud$, $h_L$, and $h_R$. The factor of two on the right-hand side originates from the definition of $\overline{w}_{\alpha,\beta}$, see Eq.~(\ref{eq:wbar}). To now form the complete weighted Sobolev norm, we take Eq.~(\ref{eq:u_wsn}) and add $\iota$ times Eq.~(\ref{eq:up_wsn}) as
\begin{equation*}
	\begin{split}
		\frac{1}{2}\dx{}{t}\rint{\Bigg(\big(\hud\big)^2 + \iota\bigg(\pxi{p}{\hud}{\zeta}\bigg)^2\Bigg)w_{\alpha,\beta}}{\zeta} = &- \frac{1}{2}\rint{\px{(\hud)^2}{\zeta}w_{\alpha,\beta}}{\zeta}\\
	 	&- (\huI_L-\hud_L)\rint{\bigg(\px{h_L\hud}{\zeta} - h_L\dx{\hud}{\zeta}\bigg)w_{\alpha,\beta}}{\zeta}\\
	 	&- (\huI_r-\hud_R)\rint{\bigg(\px{h_R\hud}{\zeta} - h_R\dx{\hud}{\zeta}\bigg)w_{\alpha,\beta}}{\zeta} \\
	 	&- 2(\huI_L-\hud_L)\iota\pxi{p}{\hud}{\zeta}\dxi{p+1}{h_L}{\zeta}\overline{w}_{\alpha,\beta}\\
		&- 2(\huI_R-\hud_R)\iota\pxi{p}{\hud}{\zeta}\dxi{p+1}{h_R}{\zeta}\overline{w}_{\alpha,\beta}.
	\end{split}
\end{equation*}
Hence, by analogy to Vincent~\etal~\cite{Vincent2010} and Eq.~(\ref{eq:osfr_cond}) in section~\ref{sec:fr}, if the following conditions are imposed on the correction function when $\hud$ is a $p^\mathrm{th}$ order polynomial
	\begin{subequations}\label{eq:h_cond}
		\begin{align}
			\rint{\bigg(h_L\dx{\hud}{\zeta}\bigg)w_{\alpha,\beta}}{\zeta} - \iota\pxi{p}{\hud}{\zeta}\dxi{p+1}{h_L}{\zeta}\rint{w_{\alpha,\beta}\,}{\zeta} &= 0, \label{eq:hl_cond}\\
			\rint{\bigg(h_R\dx{\hud}{\zeta}\bigg)w_{\alpha,\beta}}{\zeta} - \iota\pxi{p}{\hud}{\zeta}\dxi{p+1}{h_R}{\zeta}\rint{w_{\alpha,\beta}\,}{\zeta} &= 0, \label{eq:hr_cond}
		\end{align}
	\end{subequations}
	then
	\begin{equation}\label{eq:fr_norm_w_corr}
	\begin{split}
		\frac{1}{2}\dx{}{t}\rint{\Bigg(\big(\hud\big)^2 + \iota\bigg(\pxi{p}{\hud}{\zeta}\bigg)^2\Bigg)w_{\alpha,\beta}}{\zeta} = &- \frac{1}{2}\rint{\px{(\hud)^2}{\zeta}w_{\alpha,\beta}(\zeta)}{\zeta}\\
	 	&- (\huI_L-\hud_L)\rint{\bigg(\px{h_L\hud}{\zeta}\bigg)w_{\alpha,\beta}\,}{\zeta}\\
	 	&- (\huI_r-\hud_R)\rint{\bigg(\px{h_R\hud}{\zeta}\bigg)w_{\alpha,\beta}\,}{\zeta}.
	\end{split}
\end{equation}
	In order to find correction functions that can meet the conditions of Eq.~(\ref{eq:hl_cond}~\&~\ref{eq:hr_cond}) we need to define the projection of $\hud$ and $h_L$ into the Jacobi polynomial basis
\begin{equation}\label{eq:jacobi_basis}
	\hud = \sum^p_{i=0}\tilde{u}_iJ^{(\alpha,\beta)}_i \quad \mathrm{and} \quad  h_L = \sum^{p+1}_{i=0}\tilde{h}_{L,i}J^{(\alpha,\beta)}_i \quad \mathrm{and} \quad  h_R= \sum^{p+1}_{i=0}\tilde{h}_{R,i}J^{(\alpha,\beta)}_i.
\end{equation}
	Before proceeding, we will lay out some results for Jacobi polynomials that will be used throughout. Firstly, the orthogonality condition
	\begin{equation*}
		\rint{J^{(\alpha,\beta)}_mJ^{(\alpha,\beta)}_nw_{\alpha,\beta}}{\zeta} = \underbrace{\frac{2^{\alpha+\beta+1}}{2n+\alpha+\beta+1}\frac{\Gamma(n+\alpha+1)\Gamma(n+\beta+1)}{n!\Gamma(n+\alpha+\beta+1)}}_{q_n^{(\alpha,\beta)}}\delta_{mn},
	\end{equation*}
	where $\Gamma(x)$ is the gamma function and $\delta_{mn}$ is the Kronecker delta function. Secondly, it will be useful to differentiate a Jacobi polynomial and express the result as a series of Jacobi polynomials in the same basis.  From Doha~\cite{Doha2002} we have
	\begin{equation*}
		\dxi{m}{J^{(\alpha,\beta)}_n}{\zeta} = 2^{-m}\poch{n+\alpha+\beta+1}{m}\sum^{n-m}_{i=0}D_{n-m,i}(\alpha+m,\beta+m,\alpha,\beta)J^{(\alpha,\beta)}_i,
	\end{equation*}
	where
	\begin{equation}\label{eq:jacodi_D}
		\begin{split}
		D_{j,i}(\gamma,\delta,\alpha,\beta) =& \frac{\poch{j+\gamma+\delta+1}{i}\poch{i+\gamma+1}{j-i}\Gamma(i+\alpha+\beta+1)}{(j-i)!\Gamma(2i+\alpha+\beta+1)} \times \\
		&_3F_2\left(\begin{matrix}
		i-j, & j+i+\gamma+\delta+1, & i+\alpha+1 \\
		i+\gamma+1, & 2i+\alpha+\beta+2&\end{matrix};1\right).
		\end{split}
	\end{equation}
	We define here that $\poch{x}{i}$ is the rising Pochhammer function and $_3F_2(\dots;z)$ is the 3-2 generalised hypergeometric function~\cite{Bailey1933,Bateman1953}. We will also define the $p^\mathrm{th}$ derivative of a degree $p$ Jacobi polynomial as
	\begin{equation*}
		b_p^{(\alpha,\beta)} = \dxi{p}{J^{(\alpha,\beta)}_p}{\zeta} = 2^{-p}(p+\alpha+\beta+1)_p.
	\end{equation*}
	With these definitions, we may now substitute Eq.~(\ref{eq:jacobi_basis}) into Eq.~(\ref{eq:hl_cond}) to get the condition
	\begin{equation}\label{eq:jacobi_l_crit}
		\rint{\bigg(\sum^{p+1}_{i=0}\sum^{p}_{j=0}\tilde{h}_{L,i}\tilde{u}_jJ^{(\alpha,\beta)}_i\dx{J^{(\alpha,\beta)}_j}{\zeta}\bigg)w_{\alpha,\beta}}{\zeta}
		- \iota\tilde{u}_p\tilde{h}_{L,p+1}b_p^{(\alpha,\beta)}b_{p+1}^{(\alpha,\beta)}q_{0}^{(\alpha,\beta)} = 0.
	\end{equation}
	To solve this integral condition we may constrain, without loss of generality, $\tilde{h}_L$ to only have terms in $p-1$, $p$ and, $p+1$. We then observe that the only contribution to the equality in Eq.~(\ref{eq:jacobi_l_crit}) is from the $(p-1)^{\mathrm{th}}$ term of $\tilde{h}_L$. Therefore, if we can find a closed form of $D_{p-1,p-1}(\alpha+1,\beta+1,\alpha,\beta)$ we can relate $\tilde{h}_{L,p-1}$ to $\tilde{h}_{L,p+1}$ through $\iota$. Hence, substituting the values into Eq.~(\ref{eq:jacodi_D}) we find that the hypergeometric function component becomes ${_3}F_2(0,\dots;1)$ and, from the definition of the rising Pochhammer function, this must have a value of unity, hence we may write
	\begin{equation*}
		D_{p-1,p-1}(\alpha+1,\beta+1,\alpha,\beta) = \frac{(2p+\alpha+\beta-1)(2p+\alpha+\beta)}{2(p+\alpha+\beta)},
	\end{equation*}
	which leads to
	\begin{equation}\label{eq:iota_hratio}
		\iota = \frac{\tilde{h}_{L,p-1}}{\tilde{h}_{L,p+1}}\Bigg(\underbrace{\frac{(p+\alpha+\beta+1)\poch{p+\alpha+\beta+2}{p-1}q_{p-1}^{(\alpha,\beta)}}{2\poch{p+\alpha+\beta}{p-1}b_{p}^{(\alpha,\beta)}b_{p+1}^{(\alpha,\beta)}q_{0}^{(\alpha,\beta)}}}_{A_p}\Bigg).
	\end{equation}
    Using the fact that $h_L(-1) = 1$ and $h_L(1) = 0$ we find
	\begin{align*}
		(-1)^p &= \tilde{h}_{L,p}\underbrace{\bigg(\frac{\poch{\beta+1}{p}}{p!}\bigg)\bigg[\bigg(\frac{p+\alpha}{p+\beta}\bigg)\frac{\iota p(p+1)+(p+\beta)(p+\beta+1)A_p}{\iota p(p+1)+(p+\alpha)(p+\alpha+1)A_p}+1\bigg]}_{\kappa_p^{(\alpha,\beta)}}, \\
		\tilde{h}_{L,p+1} &= -\frac{A_pJ^{(\alpha,\beta)}_{p}(1)\tilde{h}_{L,p}}{\iota J^{(\alpha,\beta)}_{p-1}(1) + A_pJ^{(\alpha,\beta)}_{p+1}(1)} = -\frac{A_p(p+1)(p+\alpha)\tilde{h}_{L,p}}{\iota p(p+1) + (p+\alpha+1)(p+\alpha)A_p},
	\end{align*}
    repeating for the right and combining these results we obtain
    \begin{subequations}
		\begin{align}
			h_L &= \frac{(-1)^p}{\kappa_p^{(\alpha,\beta)}}\bigg(J^{(\alpha,\beta)}_{p} -\frac{\iota (p+1)(p+\alpha)J^{(\alpha,\beta)}_{p-1} + (p+1)(p+\alpha)A_pJ^{(\alpha,\beta)}_{p+1}}{\iota p(p+1) + (p+\alpha+1)(p+\alpha)A_p} \bigg) \label{eq:hl}, \\
			h_R &= \frac{1}{\kappa_p^{(\beta,\alpha)}}\bigg(J^{(\alpha,\beta)}_{p} +\frac{\iota (p+1)(p+\beta)J^{(\alpha,\beta)}_{p-1} + (p+1)(p+\beta)A_pJ^{(\alpha,\beta)}_{p+1}}{\iota p(p+1) + (p+\beta+1)(p+\beta)A_p} \bigg) \label{eq:hr}.
		\end{align}
	\end{subequations}
	
	We now have definitions for the left and right correction functions which lead to the decay of the weighted Sobolev norm being defined by Eq.~(\ref{eq:fr_norm_w_corr}). However, the right-hand-side of this expression is in terms of the solution and the weight function, which consequently does not permit a strict bounding on the energy decay rate. Instead to further understand the stability of this set of schemes we will turn to semi-discrete and fully-discrete Fourier analysis. 
	
	\begin{remark}
		The possibility of defining Huynh's Lumped Chebyshev--Lobatto~\cite{Huynh2007} in this set of corrections was investigated. It was suspected that this may be possible due to the ability to write LCL correction functions in Chebyshev polynomials. However, due to LCL corrections functions having a zeroth order term, they were found to not be this proposed set of correction functions. 
	\end{remark}
	
	\begin{remark}
		In previous works~\cite{Allaneau2011,Zwanenburg2016} it was found that OSFR can be defined as linearly filtered DG, with $c=0$ showing OSFR being analogous to NDG. The continuation of this property to the weighted schemes defined here was investigated and this was found to not generally hold. This arises due to the weight function and its gradient, the later in the case of OSFR is zero which permits the link to NDG. 
	\end{remark}
\subsection{Spectral Difference Schemes}\label{sec:sd}
	It has been shown~\cite{Huynh2007,Jameson2010,Vincent2010} that FR is able to recover SD schemes~\cite{Kopriva1996,Liu2006,Sun2007} for equation sets with linear and homogeneous flux functions. To obtain an SD scheme within FR the roots in the interior of the reference domain for the left and right corrections functions must be the same, i.e. $h_L(\zeta_z)=h_R(\zeta_z)=0$ for $\zeta_z\in(-1,1)$. 
	
	The simplest method of achieving an SD in FR is by prescribing the interior zeros to be a set of quadrature points. In the first proof of energy stability of SD schemes the canonical Gauss--Legendre quadrature~\cite{Jameson2010} was utilised, but other quadratures have been considered. For example Lui~\etal~\cite{Liu2006} saw the potential for a Gauss--Legendre--Lobatto quadrature. As the choice of quadrature is arbitrary, we proposed that this be extended to the full set of Gauss--Jacobi quadratures, where the $p^{\mathrm{th}}$ Gauss--Jacobi points are the roots of the $p^{\mathrm{th}}$ order Jacobi polynomial~\cite{Kress1998}. These give rise to the correction functions
	\begin{equation}\label{eq:jacobi_sd}
		h_{L,\mathrm{SD}} = \frac{(1-\zeta)}{2}\frac{J^{(\alpha,\beta)}_p(\zeta)}{J^{(\alpha,\beta)}_p(-1)} \quad \mathrm{and} \quad h_{R,\mathrm{SD}} = \frac{(1+\zeta)}{2}\frac{J^{(\alpha,\beta)}_p(\zeta)}{J^{(\alpha,\beta)}_p(1)},
	\end{equation}
	where the correction functions are normalised for the edge value as Jacobi polynomials do not guarantee unit magnitude at $\pm1$.	The recurrence relation for Jacobi polynomials \cite{Beals2016} leads us to the following form
	\begin{subequations}
		\begin{align}
			h_{L,\mathrm{SD}} &= (-1)^p\gamma_p(\beta)\bigg(-a_pJ^{(\alpha,\beta)}_{p-1} + \frac{1}{2}(1 - b_p)J^{(\alpha,\beta)}_p - c_pJ^{(\alpha,\beta)}_{p+1}\bigg), \\
			h_{R,\mathrm{SD}} &= \gamma_p(\alpha)\bigg(a_pJ^{(\alpha,\beta)}_{p-1} + \frac{1}{2}(1 + b_p)J^{(\alpha,\beta)}_p + c_pJ^{(\alpha,\beta)}_{p+1}\bigg).
		\end{align}
	\end{subequations}
	Where
	\begin{subequations}
		\begin{align}
			a_p &=  \frac{(p+\alpha)(p+\beta)}{(2p+\alpha+\beta+1)(2p+\alpha+\beta)}, \\
			b_p & = -\frac{\alpha^2-\beta^2}{(2p+\alpha+\beta+2)(2p+\alpha+\beta)}, \\
			c_p &=  \frac{(p+1)(p+1+\alpha+\beta)}{(2p+\alpha+\beta+1)(2p+\alpha+\beta+2)}, \\
			\gamma_p(x) &= \frac{\Gamma(p+1)\Gamma(x+1)}{\Gamma(p+x+1)}.
		\end{align}
	\end{subequations}
	
	 This is included for completeness, however the focus will be put on the Jacobi SD function of Eq.(\ref{eq:jacobi_sd}). It is also noteworthy that this definition of SD schemes may be extended to include all subsets of Gauss--Jacobi sparse grids by defining the correction function as
	\begin{equation}
		h_L = \frac{(1-\zeta)}{2}\frac{J^{(\alpha,\beta)}_m(\zeta)}{J^{(\alpha,\beta)}_m(-1)} \quad \mathrm{for} \quad |m|_1=p,
	\end{equation}
	where $m=(m_1,m_2\dots)$ is a multi-index and $(\alpha,\beta)$ may vary with $m_i$.

	\begin{figure}[tbhp]
		\centering
			\resizebox{0.7\linewidth}{!}{\savebox\mybox{\begin{tikzpicture}
		\begin{axis}[xtick={-1,-0.5,0,0.5,1},
				xmin=-1,xmax=1,ymin=0,ymax=0.1,
  				style={font=\Large},axis x line=middle,axis y line=none]
				\addplot[color={RdBu-C},mark=*,style={very thick}] coordinates {(-0.86114,0.01) (-0.33998,0.01) (0.33998,0.01) (0.86114,0.01)};	
				\node[] at (axis cs: 0,0.02) {$\alpha=\beta=0$};
		\end{axis}
		\draw [ultra thick, draw=black,draw opacity=0]
       (0,0) -- (0,-5);
	\end{tikzpicture}}

	\begin{tikzpicture}
		\begin{groupplot}[group style = {group size = 2 by 2,
			horizontal sep = 120pt, vertical sep = 0pt}]
			\nextgroupplot[xtick={-1,-0.5,0,0.5,1},
				xmin=-1,xmax=1,ymin=0,ymax=0.1,
				style={font=\Large},axis x line=middle,axis y line=none]
				\addplot[color={RdBu-C},mark=*,style={very thick}] coordinates {(0.13929,0.01) (0.57599,0.01) (0.86945,0.01) (0.99964,0.01)};
				\node[] at (axis cs: 0,0.02) {$\alpha=-1, \beta = 10$};
				
			\nextgroupplot[,xtick={-1,-0.5,0,0.5,1},
				xmin=-1,xmax=1,ymin=0,ymax=0.1,
  				style={font=\Large},axis x line=middle,axis y line=none]
				\addplot[color={RdBu-C},mark=*,style={very thick}] coordinates {(-0.44721,0.01) (-0.14907,0.01) (0.14907,0.01) (0.44721,0.01)};
				\node[] at (axis cs: 0,0.02) {$\alpha=\beta=10$};
			
			\nextgroupplot[xtick={-1,-0.5,0,0.5,1},
				xmin=-1,xmax=1,ymin=0,ymax=0.1,
  				style={font=\Large},axis x line=middle,axis y line=none]
				\addplot[color={RdBu-C},mark=*,style={very thick}] coordinates {(-0.99834,0.01) (-0.44558,0.01) (0.44558,0.01) (0.99834,0.01)};
				\node[] at (axis cs: 0,0.02) {$\alpha=\beta=-1$};
				
			\nextgroupplot[xtick={-1,-0.5,0,0.5,1},
				xmin=-1,xmax=1,ymin=0,ymax=0.1,
  				style={font=\Large},axis x line=middle,axis y line=none]
				\addplot[color={RdBu-C},mark=*,style={very thick}] coordinates {(-0.99964,0.01) (-0.86945,0.01) (-0.57599,0.01) (-0.13929,0.01)};	
				\node[] at (axis cs: 0,0.02) {$\alpha=10, \beta = -1$};
				
		\end{groupplot}
		\node at ($(group c2r1) + (-5.6,-5.5)$) {\usebox\mybox};
	\end{tikzpicture}}
		\caption{Selected Gauss-Jacobi quadratures for $n=4$.}
		\label{fig:gj_quad}
	\end{figure}
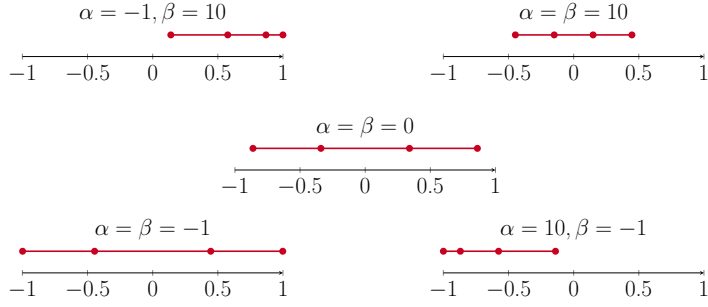
	
	To briefly touch on some characteristics of the Jacobi SD correction functions, consider the Gauss--Jacobi quadratures that make up the interior zeros of the correction functions. Figure~\ref{fig:gj_quad} aims to demonstrate how the quadrature is effected by the variation of $\alpha$ and $\beta$---primarily that for $\alpha=\beta$ the quadrature is symmetric. Furthermore, in the limit $\alpha,\beta\rightarrow-1$ the quadrature gets pinned to the edges and as $\alpha,\beta\rightarrow\infty$ the quadrature gets compressed to the middle. The effect this has for the case when $\alpha\neq\beta$ is that the left and right correction functions are not symmetric, with the zeros biased towards one of the edges.

\subsection{Limits on Stability}
\label{sec:limits}

In order for the correction function to be stable,  Eq.~(\ref{eq:weighted_sobolev_decay}) must be satisfied.  Moreover, the norm described by Eq.~(\ref{eq:weighted_sobolev}) must indeed be a norm; \emph{i.e.} it must be positive, definite, homogeneous, and obey the triangle inequality. Of particular concern to us is the positivity condition; for the remaining conditions follow immediately due to the linear nature of differentiation.  Hence, it is required that
	\begin{equation}\label{eq:norm_positive}
		0 < \|\hud\|^2_{W_2^{\iota,w}} = \rint{\Bigg(\big(\hud\big)^2 + \iota\bigg(\pxi{p}{\hud}{\zeta}\bigg)^2\Bigg)(1-\zeta)^{\alpha}(1+\zeta)^{\beta}}{\zeta} < \infty,
	\end{equation}
	for $\hud\neq0$.
	Therefore, following the method of \cite{Vincent2010} we can substitute Eq.~(\ref{eq:jacobi_basis}) into Eq.~(\ref{eq:norm_positive}) as
	\begin{equation}
		0 < \sum^p_{i=0}\tilde{u}^2_iq_i^{(\alpha,\beta)} + \iota \big(b_p^{(\alpha,\beta)}\big)^2q_0^{(\alpha,\beta)} < \infty.
	\end{equation}
	Then grouping terms of the same order
	\begin{equation}
		0 < \sum^{p-1}_{i=0} \tilde{u}_i^2q_i^{(\alpha,\beta)} + \bigg(q_p^{(\alpha,\beta)}+\iota \big(b_p^{(\alpha,\beta)}\big)^2q_0^{(\alpha,\beta)}\bigg)\tilde{u}_p^2 < \infty.
	\end{equation}
	As $\tilde{u}_i^2$ is always positive, the limit on the value of $\iota$ is
	\begin{equation}\label{eq:i_crit}
		-\!\iota_{\mathrm{crit}}=-\frac{q_p^{(\alpha,\beta)}}{\big(b_p^{(\alpha,\beta)}\big)^2q_0^{(\alpha,\beta)}} \leqslant \iota < \infty.
	\end{equation}
	The value of $\iota_{\mathrm{crit}}$ can be evaluated in a closed form that shows it is always positive. Consider
	\begin{equation}
		\begin{split}
			\iota_{\mathrm{crit}} &= \frac{q_p^{(\alpha,\beta)}}{\big(b_p^{(\alpha,\beta)}\big)^2q_0^{(\alpha,\beta)}} \\
			&= \bigg(\frac{\alpha+\beta+1}{2p+\alpha+\beta+1}\bigg)\bigg(\frac{\poch{\alpha+1}{p}\poch{\beta+1}{p}}{\poch{\alpha+\beta+1}{p}}\bigg)\times \\
			\bigg(\frac{2^p}{\poch{p+\alpha+\beta+1}{p}}\bigg)^2.
		\end{split}
	\end{equation}
	A consequence of this is that the norm is always valid for $\iota=0$, which is of particular interest due to the norms' analogy to that of nodal DG.
	
	To show this definition of $\iota_\mathrm{crit}$ is consistent with previous work, let us consider $\alpha=\beta=0$. Remembering that $\poch{1}{p} = p!$, we get
	\[
		\frac{q_p^{(0,0)}}{\big(b_p^{(0,0)}\big)^2q_0^{(0,0)}} = \frac{(p!)^3}{(2p+1)}\bigg(\frac{2^p}{(2p)!}\bigg)^2,
	\]
	 which, noting that $c = 2\iota$, is equivalent to the result presented in section~\ref{sec:osfr}.
	
	In section~\ref{sec:sd} the set of SD correction function was extended using Gauss--Jacobi quadratures. The advantage of this method of extending the set of SD correction functions is that we may use Eq.~(\ref{eq:i_crit}) in order to understand if SD provides a valid norm. We therefore assert:
	\begin{lemma}{SD Energy Stability.}
		An SD correction function with $\alpha,\beta\in(-1,\infty)$ will always give the $\|\hud\|_{W_2^{\iota,w}}$ to be positive and hence is a valid norm.
	\end{lemma}
	\begin{proof}
		Using Eq.(\ref{eq:iota_hratio}) to produce an expression of $\iota$ for SD schemes, which, after normalisation by $\iota_{\mathrm{crit}}$, gives
	\begin{equation}
		\frac{\iota_{\mathrm{SD}}}{\iota_{\mathrm{crit}}} = \frac{p}{p+1}.
	\end{equation}
	As $\iota_{\mathrm{crit}}$ is always positive, so must $\iota_{\mathrm{SD}}$ always be positive.
	\end{proof}
\section{Fourier Analysis}\label{sec:conv}
	In section~\ref{sec:stability}, two conditions were imposed, Eq.~(\ref{eq:h_cond}), which were analogous to those of Vincent~\etal~\cite{Vincent2010}. When the weight function was the unit measure, as it is for OSFR, this led to the norm being solution independent. However, as can be seen from Eq.~(\ref{eq:fr_norm_w_corr}), the norm is not solution independent. The extra conditions on the correction function to remove this dependency leads to an over constraint of the correction function. Therefore, proof of the stability of this set of correction functions has to shown through another means, in particular via Fourier analysis of the linear advection equation.
\subsection{Semi-Discrete}	
		The semi-discrete Fourier analysis of other FR schemes has been performed in several works~\cite{Huynh2007,Vincent2011,Asthana2014,Vermeire2016a} and we begin with the semi-discretised matrix form of FR applied to linear advection
		\begin{equation}\label{eq:semi_disc}
			\px{\mathbf{u}_j}{t} = -\frac{2}{h}\Big(\mathbf{C}_+\mathbf{u}_{j+1} + \mathbf{C}_0\mathbf{u}_j + \mathbf{C}_-\mathbf{u}_{j-1} \Big).
		\end{equation}
		Here we have assumed a uniform mesh with grid spacing $h$.	The operator matrices are defined as
		\begin{subequations}
			\begin{align}
				\mathbf{C}_+ &= (1-\kappa)\mathbf{g}_R\mathbf{l}_L^T, \\
				\mathbf{C}_0 &= \mathbf{D} - \kappa\mathbf{g}_L\mathbf{l}_L^T - (1-\kappa)\mathbf{g}_R\mathbf{l}_R^T, \\
				\mathbf{C}_- &= \kappa\mathbf{g}_L \mathbf{l}_R^T,
			\end{align}
		\end{subequations}
		where $\kappa$ is the upwinding ratio (fully upwinded and central differenced for $\kappa=1$ and $0.5$ respectively) and $\mathbf{D}$ is the differentiation matrix. The array $\mathbf{l_L}$  interpolates from the solution points to the left interface and $\mathbf{l_R}$ is similarly defined. $\mathbf{g_L}$ and $\mathbf{g_R}$ are values of the left and right correction function gradients at the solution points. If we now impose a solution of the form
		\begin{equation}\label{eq:f_soln}
			\mathbf{u}_j = \exp{(i(k\mathbf{x}_j-\omega t))} = \exp{\bigg(ik\bigg[\frac{h}{2}(\pmb{\zeta}+1) + x_j - ct\bigg]\bigg)},
		\end{equation}
		applying this to Eq.~(\ref{eq:semi_disc}) results in
		\begin{equation}
			\px{\mathbf{u}_j}{t} = -\underbrace{\frac{2}{h}\Big(\mathbf{C}_+\exp{(ikh)} + \mathbf{C}_0 + \mathbf{C}_-\exp{(-ikh)}\Big)}_\mathbf{Q}\mathbf{u_j}.
		\end{equation}
		Differentiating $\mathbf{u}_j$ with respect to time we find
		\begin{equation}
			c\mathbf{u}_j = -\frac{2i}{hk}\Big(\mathbf{C}_+\exp{(ikh)} + \mathbf{C}_0 + \mathbf{C}_-\exp{(-ikh)}\Big)\mathbf{u}_j.
		\end{equation}
		This is clearly an eigenvalue problem, where $c$ is the modified phase speed ($\omega/k$) of the scheme and for a given wavenumber, $k$, if $\Im{(c)}\leqslant0$ then the solution will not grow with time. Stipulating that the solution is a plane wave is important as for a periodic or infinite domain, all continuous solutions to the linear advection equation may be written as the superposition of plane waves. Therefore, the necessary and sufficient condition on scheme stability can be written as $\Im{(c)}\leqslant 0 \;\forall\; k\in[0,k_{\mathrm{nq}}]$.
		
	\begin{figure}
		\centering
		\begin{subfigure}[b]{0.48\linewidth}
			\centering
			\includegraphics[width=\linewidth]{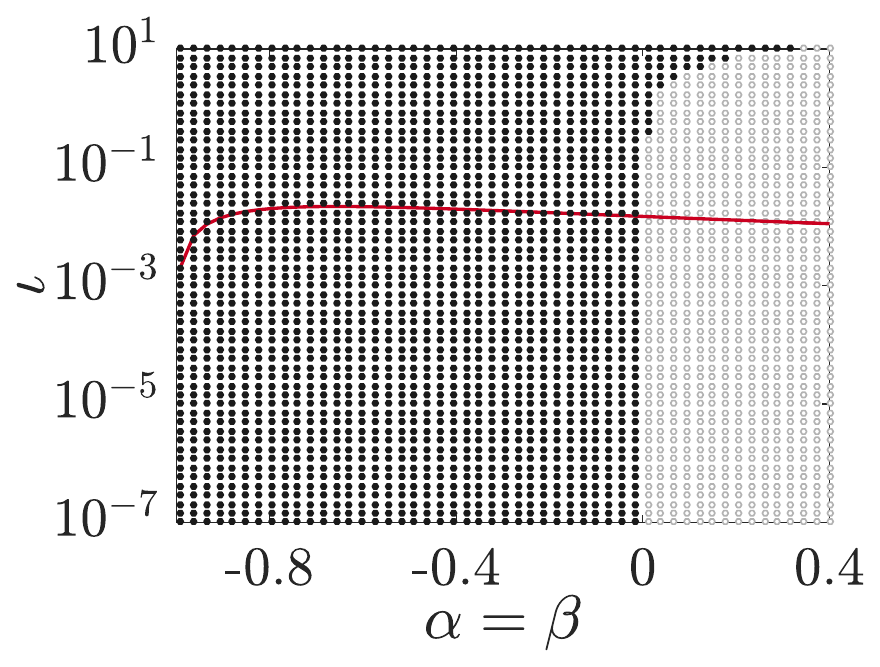}
			\caption{\label{fig:FR2_up_stab}$p=2$.}
		\end{subfigure}
		~
		\begin{subfigure}[b]{0.48\linewidth}
			\centering
			\includegraphics[width=\linewidth]{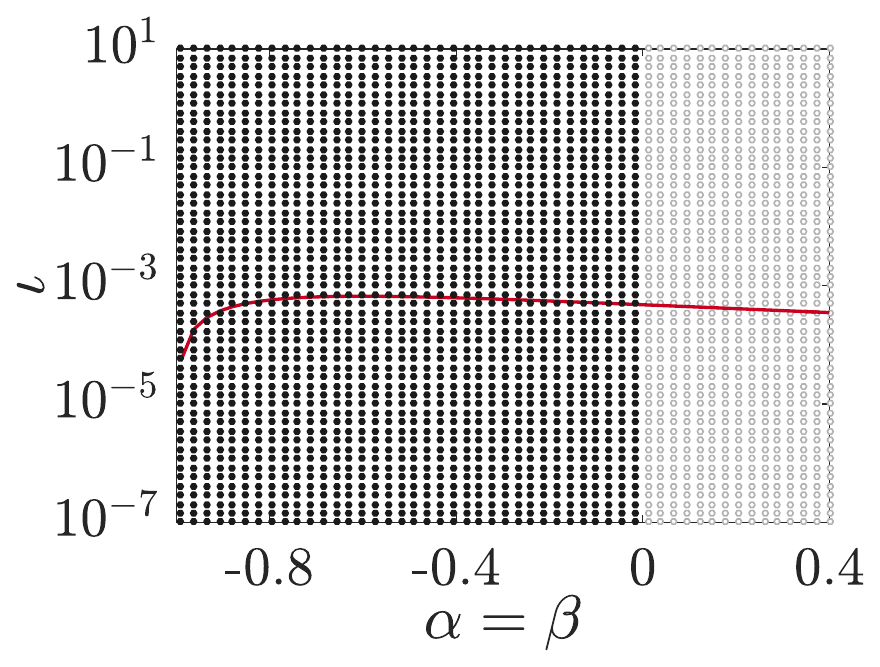}
			\caption{\label{fig:FR3_up_stab}$p=3$.}
		\end{subfigure}
		~
		\begin{subfigure}[b]{0.48\linewidth}
			\centering
			\includegraphics[width=\linewidth]{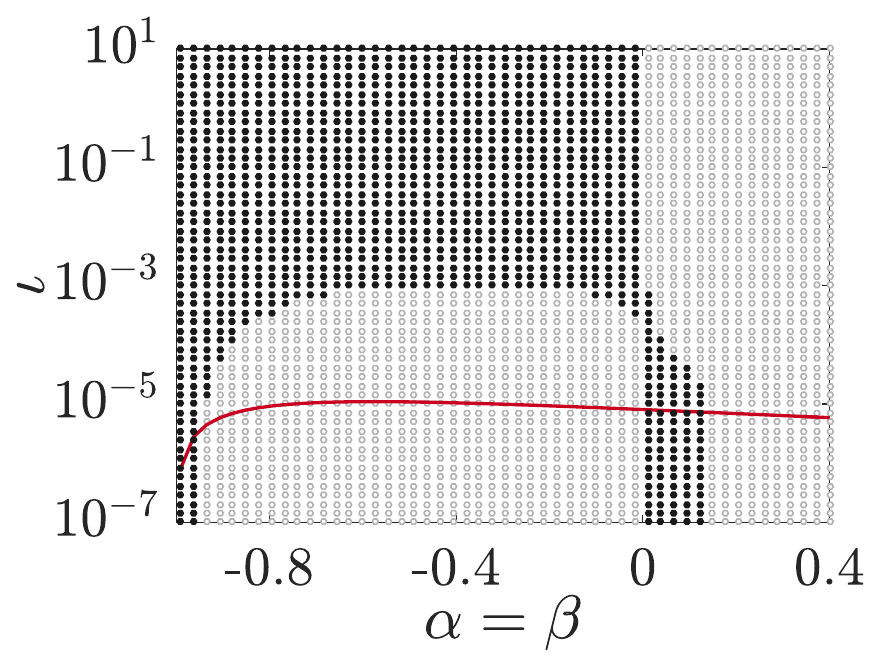}
			\caption{\label{fig:FR4_up_stab}$p=4$}
		\end{subfigure}
		~
		\begin{subfigure}[b]{0.48\linewidth}
			\centering
			\includegraphics[width=\linewidth]{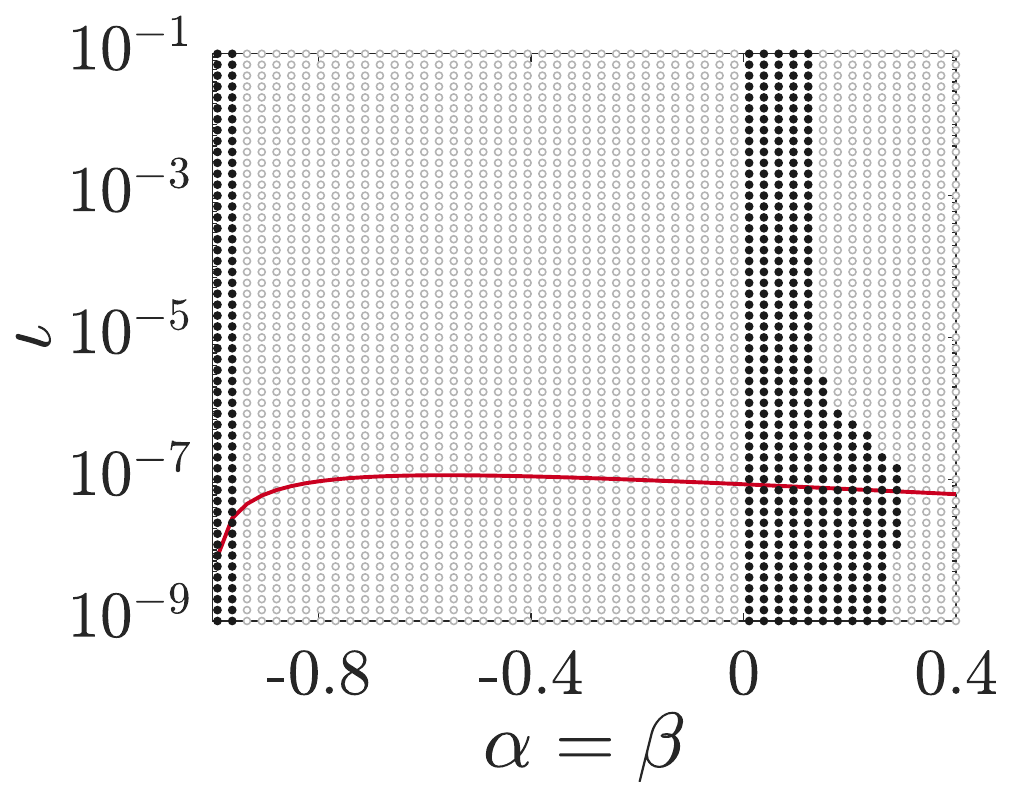}
			\caption{\label{fig:FR5_up_stab}$p=5$.}
		\end{subfigure}
		\caption{Regions of stable correction functions for $\alpha=\beta$. Filled black dots indicate stable correction functions, and the line indicates $\iota_\mathrm{SD}$.}
		\label{fig:FR_up_stab}
	\end{figure}
			
	The results of an investigation into the semi-discrete stability of the proposed schemes is shown in Fig.~\ref{fig:FR_up_stab}. Here symmetric correction functions were tested, $\alpha=\beta$, for various values of $\iota$, and filled circle in Fig.~\ref{fig:FR_up_stab} shows that the scheme was found to satisfy the semi-discrete stability condition. In particular, for all the $\iota$ values tested at $p\in\{2,3\}$ the schemes were stable for $\alpha=\beta\leqslant 0$. While, for $p\in\{4,5\}$ the stable region is reduced, the reason for which is not currently understood. Furthermore, these figures show schemes which are stable for both upwind and centrally differenced interfaces, but for $p\in\{2,3\}$ all schemes tested were stable for centrally differenced interfaces. However, $p\in\{4,5\}$ showed many schemes were stable, with some exception that did not seem to fit in a clearly defined region in the $\iota-\alpha$ plane.
	
	We now wish to understand the semi-discrete rate of convergence of the methods. Returning to the semi-discretised form of the linear advection equation for plane waves with wavenumber $k$ and unit wave speed, 
	\begin{equation}
		\px{\bud_j}{t} = -\mathbf{Q}(k)\bud_j.
	\end{equation}
	The FR operator matrix may now be diagonalised
	\begin{equation}
		\mathbf{Q} = ik\mathbf{W\Lambda W}^{-1},
	\end{equation}
	where $\mathbf{W}$ is the eigenvector matrix and $\mathbf{\Lambda}=\mathrm{diag}(\lambda_0\dots\lambda_p)$. This diagonalisation is possible in this case so long as $\mathbf{Q}$ is not rank deficient and, with non-singular solution point choice, this implies the basis of $\mathbf{Q}$ is in $\mathbb{P}_p$. We may then form the initial interpolation of the solution as
	\begin{equation}
		\bud_j(t=0) = \exp{(ikx_j)}\mathbf{Wv}_0 = \exp{\Big(ik\big(J_j(\zeta+1) +x_j\big)\Big)}.
	\end{equation}
	To calculate the rate of convergence, we need to monitor the semi-discrete error of the solution for different grid spacings. The derivation of the analytical semi-discrete and fully-discrete error can be found from Asthana~\etal~\cite{Asthana2017} and Trojak~\etal~\cite{Trojak2018c} and, as the derivation is secondary to the aim, here we will jump to the end result for the semi-discrete error
	\begin{equation}
		\mathbf{e}_j(t,J) = \bud_j(t) - \mathbf{u}_j(t) = \exp{(ik(x_j-t))} \sum^p_{n=0}\bigg(\exp{(ikt(\lambda_n+1))}-1\bigg)\mathbf{v}_{0,n}\mathbf{w}_n.
	\end{equation}
	Here $\mathbf{w}_n$ is the $n^{\mathrm{th}}$ column vector of $\mathbf{W}$. If we define the $\ell_2$ norm of the error as $E_j(t,J) = \|\mathbf{e}(t,J)_j\|_2$ then the grid convergence rate is
	\begin{equation}
		r_h(t) = \frac{\log{(E_j(t,J_1))}-\log{(E_j(t,J_2))}}{\log{J_1}-\log{J_2}},
	\end{equation}
	where $J_1$ and $J_2$ are the Jacobians of the two grids over which the convergence rate is to be calculated.

	\begin{figure}
		\centering
		\begin{subfigure}[b]{0.48\linewidth}
			\centering
			\includegraphics[width=\linewidth]{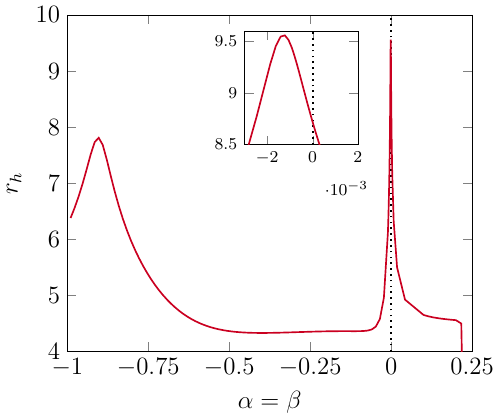}
			\caption{Upwinded interfaces.}
			\label{fig:FRJDG4_conv_up_ab}
		\end{subfigure}
		~
		\begin{subfigure}[b]{0.48\linewidth}
			\centering
			\includegraphics[width=\linewidth]{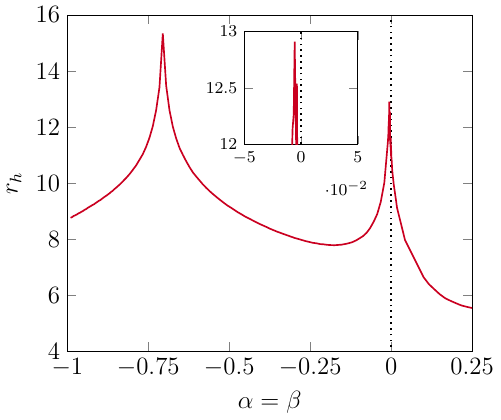}
			\caption{Central differenced interfaces.}
			\label{fig:FRJDG4_conv_cd_ab}
		\end{subfigure}
		\caption{Variation of the rate of error convergence with grid for quasi-DG correction functions,  ($\iota=0$) when, $p=4$, $\alpha = \beta$, $J_2/J_1 = 0.5$, $k=3\pi/4$ and, $t/T=1000$. The dotted line is for $\alpha=\beta=0$.}
		\label{fig:FRJDG4_conv_ab}
	\end{figure}

	\begin{figure}[tbhp]
		\centering
		\begin{subfigure}[b]{0.485\linewidth}
			\centering
			\includegraphics[width=\linewidth]{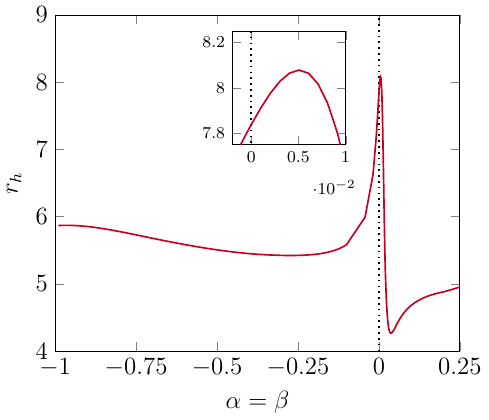}
			\caption{Upwinded interfaces.}
			\label{fig:FRJSD4_conv_up_ab}
		\end{subfigure}
		~
		\begin{subfigure}[b]{0.485\linewidth}
			\centering
			\includegraphics[width=\linewidth]{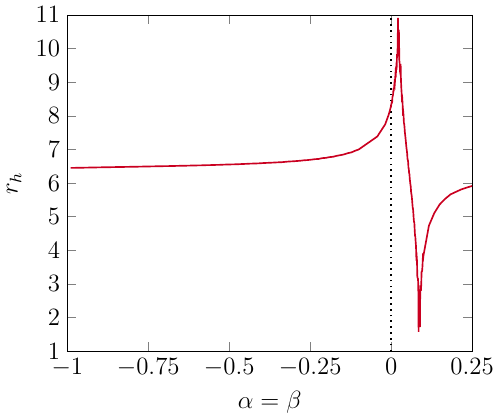}
			\caption{Central differenced interfaces.}
			\label{fig:FRJSD4_conv_cd_ab}
		\end{subfigure}
		\caption{Variation of the rate of error convergence with grid for SD correction functions when, $p=4$, $\alpha = \beta$, $J_2/J_1 = 0.5$, $k=\pi/2$ and, $t/T=1000$. The dotted line is for $\alpha=\beta=0$.}
		\label{fig:FRJSD4_conv_ab}
	\end{figure}
	
	The property that sets the DG correction functions apart from other correction functions is that it achieves super-convergence of the $\ell_2$ error at the solution points~\cite{Huynh2007,Asthana2017,Cockburn1999,Adjerid2002,Witherden2014} of order $2p+1$, \emph{i.e.} for sufficiently smooth and well resolved specific cases, the rate of error convergence can be of order $2p+1$, whereas other FR schemes would normally be expected to obtain less.

	On account of this feature, we will initially look at the rate of convergence of the qDG, with a focus on the case when $\alpha=\beta$, \emph{i.e.} symmetric weight functions. There are two rates of convergence that can be considered, the initial rate as $t\rightarrow0$ and the long time rate as $t\rightarrow\infty$, with the initial rate being dictated by the secondary modes and the long time rate by the primary mode~\cite{Trojak2018c,Asthana2017}. Due to the secondary modes having very short half-lives at a well-resolved wavenumbers~\cite{Trojak2018c}, we consider the more important convergence rate to be $r_h(t) \rightarrow \infty$, and this is shown in Figure~\ref{fig:FRJDG4_conv_ab}. It is clear that for both centrally differenced and upwinded interfaces there is a sharp drop off in the rate of convergence as $\alpha=\beta$ deviates from $0$. However there is a small region around zero where the rate of convergence is maintained. This indicates that correction function may be found that can improve the dispersion/dissipation properties of the method while also maintaining the rate of convergence. The uncharacteristically high rate of convergence at $\alpha=\beta\approx-0.7$ is found for a scheme that was marginally unstable and is disregarded.

	\begin{figure}[tbhp]
		\centering
		\begin{subfigure}[b]{0.485\linewidth}
			\centering
			\includegraphics[width=\linewidth]{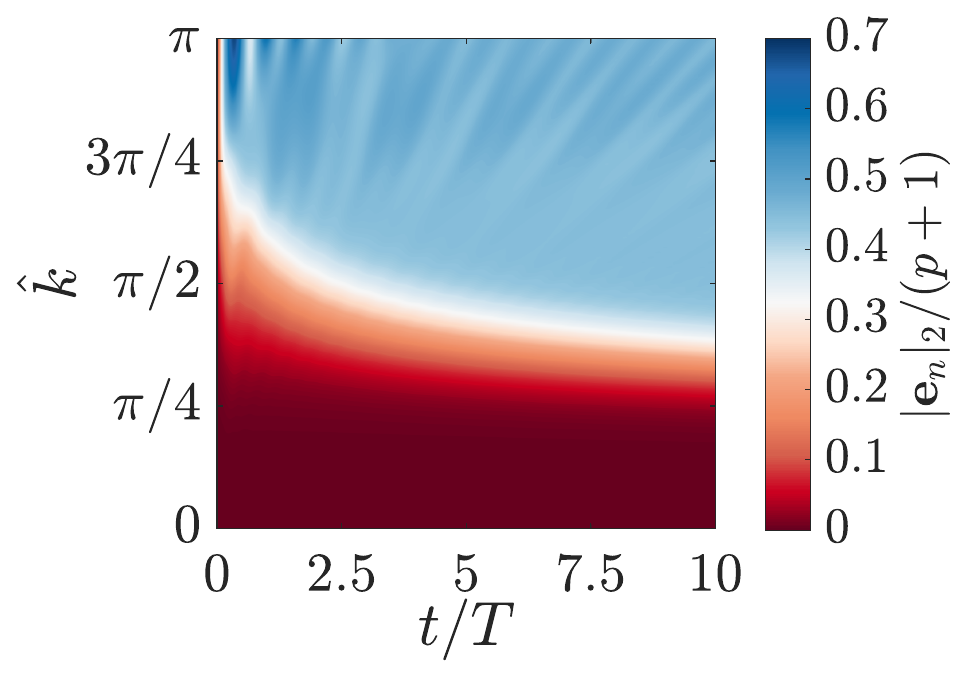}
			\caption{Upwinded interfaces \\ $\alpha=\beta=5\times10^{-3}$}
			\label{fig:FRJSD4_error_up}
		\end{subfigure}
		~
		\begin{subfigure}[b]{0.487\linewidth}
			\centering
			\includegraphics[width=\linewidth]{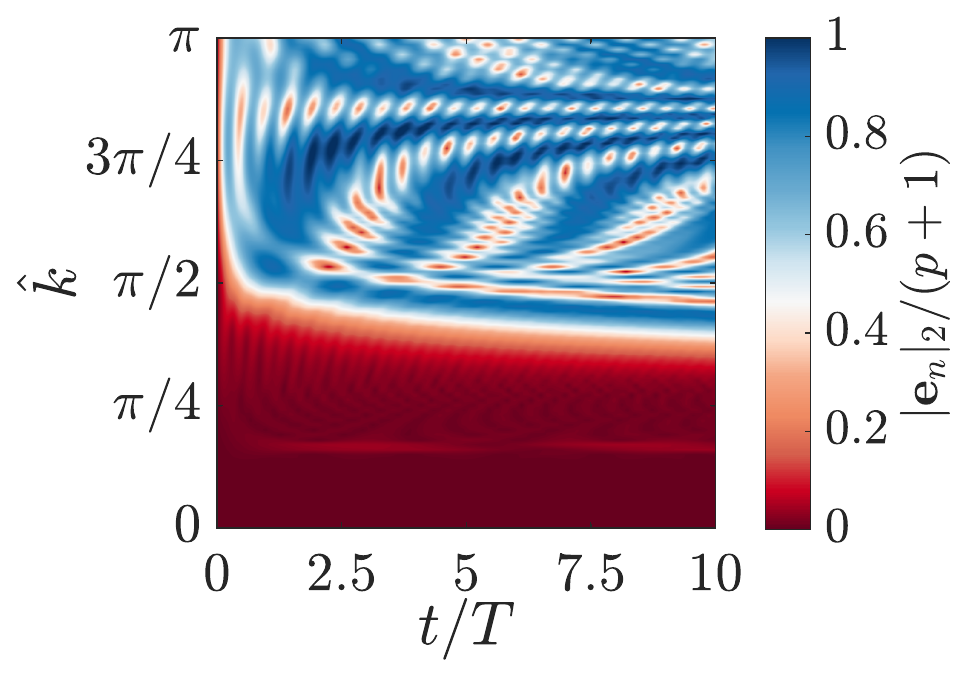}
			\caption{Central differenced interfaces \\ $\alpha=\beta=5\times10^{-2}$}
			\label{fig:FRJSD4_error_cd}
		\end{subfigure}
		\caption{Variation of error with time and normalised wavenumber ($\hat{k}=k/(p+1)$) for SD correction functions with optimal convergence from Figure~\ref{fig:FRJSD4_conv_ab}, $p=4$, $\Delta_x = 1$.}
		\label{fig:FRJSD4_opt_error}
	\end{figure}

	The investigation of convergence rate was repeated for symmetric SD correction functions and the results are shown in Fig.~\ref{fig:FRJSD4_conv_ab}. Again  a sharp drop-off in the rate is seen around $\alpha=\beta=0$, but in the case of centrally differenced interface, a large increase in the rate of convergence is seen prior to this drop off. The exact reason for this large increase is unknown, but Fig.~\ref{fig:FRJSD4_opt_error} makes it clear that this rate does appear to be serendipitous as the error at $k=\pi/2$---equivalent to $\hat{k}=\pi/10$---is low and monotonic. Therefore, an SD correction function has been found that has improved dispersion and dissipation characteristics, while maintaining theoretical super-convergence.

	
	\begin{figure}[tbhp]
		\centering
		\begin{subfigure}[b]{0.48\linewidth}
			\centering
			\includegraphics[width=\linewidth]{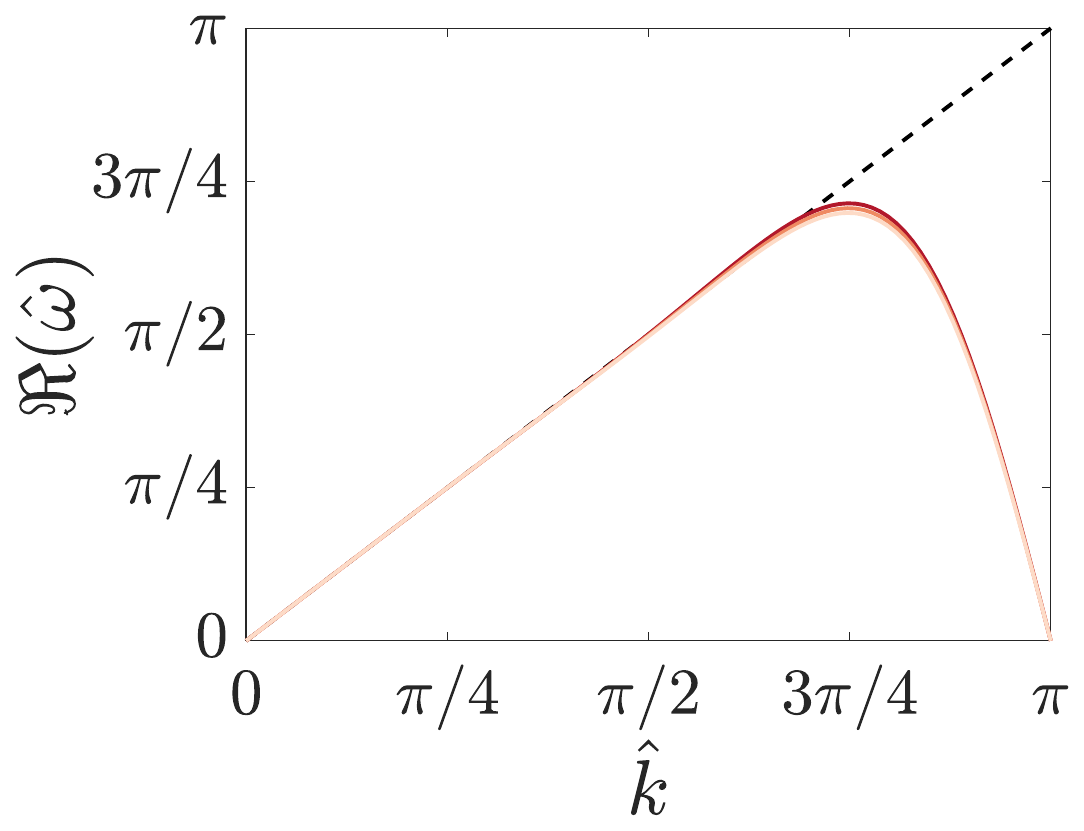}
			\caption{Dispersion}
			\label{fig:FR4JSD_R}
		\end{subfigure}
		~
		\begin{subfigure}[b]{0.48\linewidth}
			\centering
			\includegraphics[width=\linewidth]{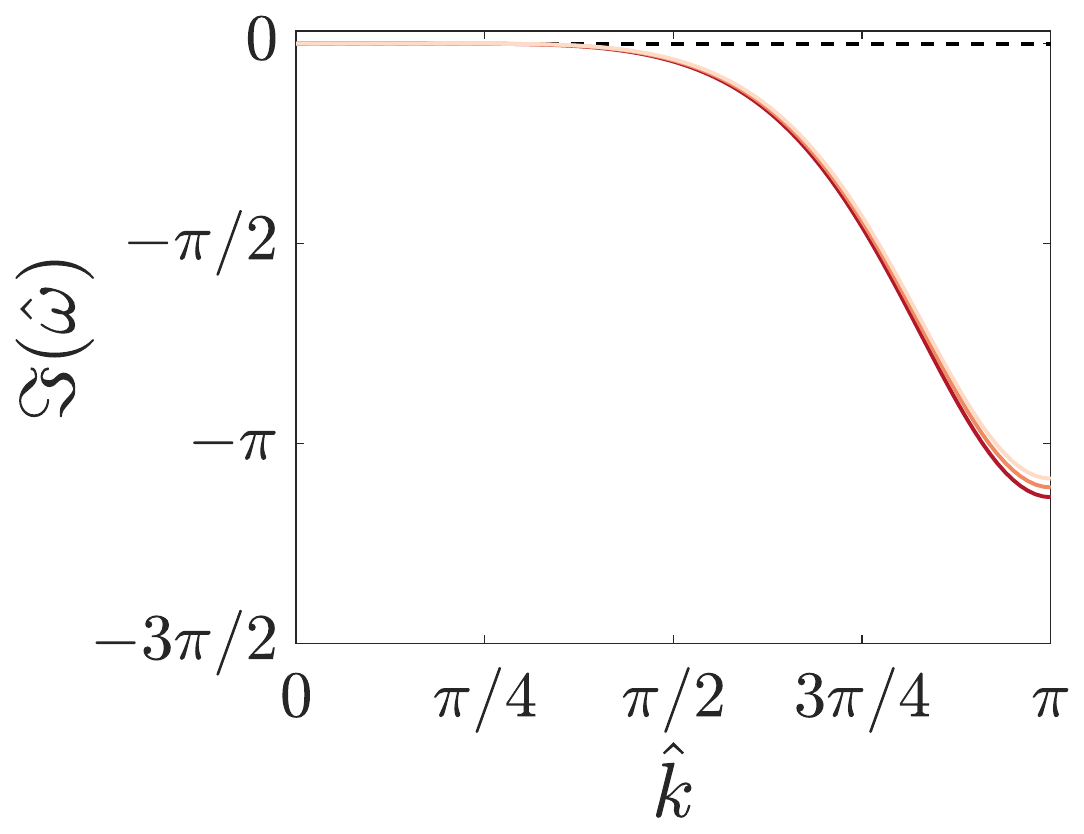}
			\caption{Dissipation}
			\label{fig:FR4JSD_I}
		\end{subfigure}
		~
		\begin{subfigure}[b]{0.27\linewidth}
			\centering
			\includegraphics[width=\linewidth]{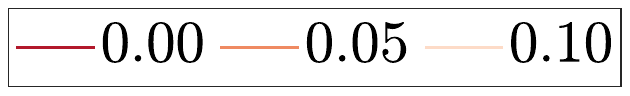}
		\end{subfigure}
		\caption{Jacobi SD dispersion/dissipation plots for $p=4$ with various values of $\alpha=\beta$. }
		\label{fig:FR4JSD_dd}
	\end{figure}

	To gain further insight into the effect of these weighted correction functions, we have plotted the dispersion and dissipation relation as $\alpha=\beta$ is varied, Fig.~\ref{fig:FR4JSD_dd}. Here the increase in $\alpha=\beta$ leads to both a reduction in the dispersion overshoot and a reduction in dissipation. Plotting the form of the correction functions as $\alpha=\beta$ was varied, Fig.~\ref{fig:FRJSDp4_corrs}, shows that increasing the $\alpha=\beta$ reduces the magnitude of the correction function and its gradient within the reference domain. This is consistent with  some of the ideas of Huynh~\cite{Huynh2007}, where smoother interior correction functions were postulated to be more stable. When the weight function associated with this change is considered, this is equivalent to weighting the domain interior more than the edges.
	
	\begin{figure}
		\centering
		\begin{subfigure}[b]{0.48\linewidth}
			\centering
				\resizebox{\linewidth}{!}{	\begin{tikzpicture}
		\pgfplotsset{
 			cycle list={
        		{RdBu-A,thick},
        		{RdBu-B,thick},
        		{RdBu-C,thick},
        		{RdBu-D,thick},
        		{RdBu-E,thick},
        		{RdBu-F,thick},
        		{RdBu-G,thick},
        		{RdBu-H,thick},
        		{RdBu-I,thick},
        		{RdBu-J,thick},
        		{RdBu-K,thick},
        		{RdBu-L,thick},
        		{RdBu-M,thick},
    		}
		}
		
		\begin{axis}[name=plot1,xlabel={$\zeta$},ylabel={$h_L$},
		    xtick={-1,-0.5,...,1},ytick={-4,-2,...,4},
		    xmin=-1,xmax=1,
		    ylabel style={rotate=-90},
    		y tick label style={
        		/pgf/number format/.cd,
            	fixed,
            	fixed zerofill,
            	precision=1,
        	/tikz/.cd
    		},
    		ymin=-4,ymax=4,
    		style={font=\large}]
    			
			\foreach \column in {1,...,13}{
				\addplot+[] table[x index={0},y index={\column},col sep=comma,unbounded coords=jump] {./Figs/data/FRJSDp4_hl_negative.csv};
			}
			\node (source) at (axis cs:0.6,0.75){$0$};
       		\node (destination) at (axis cs:0.7,-3){$-1$};
       		\draw[very thick,->,>=stealth](source)--(destination);
			
		\end{axis} 		
	\end{tikzpicture}}
			\caption{$\alpha=\beta \in (-1,0]$.}
			\label{fig:JSDp4_neg}
		\end{subfigure}
		~
		\begin{subfigure}[b]{0.48\linewidth}
			\centering
				\resizebox{\linewidth}{!}{	\begin{tikzpicture}
		\pgfplotsset{
 			cycle list={
        		{RdBu-A,thick},
        		{RdBu-B,thick},
        		{RdBu-C,thick},
        		{RdBu-D,thick},
        		{RdBu-E,thick},
        		{RdBu-F,thick},
        		{RdBu-G,thick},
        		{RdBu-H,thick},
        		{RdBu-I,thick},
        		{RdBu-J,thick},
        		{RdBu-K,thick},
        		{RdBu-L,thick},
        		{RdBu-M,thick},
    		}
		}
		
		\begin{axis}[name=plot1,xlabel={$\zeta$},ylabel={$h_L$},
		    xtick={-1,-0.5,...,1},ytick={-0.75,-0.5,...,1},
		    xmin=-1,xmax=1,
		    ylabel style={rotate=-90},
    		y tick label style={
        		/pgf/number format/.cd,
            	fixed,
            	fixed zerofill,
            	precision=1,
        	/tikz/.cd
    		},
    		ymin=-0.75,ymax=1.02,
    		style={font=\large}]
    			
			\foreach \column in {1,...,13}{
				\addplot+[] table[x index={0},y index={\column},col sep=comma,unbounded coords=jump] {./Figs/data/FRJSDp4_hl_positive.csv};
			}
			\node[anchor=west] (source) at (axis cs:0.4,0.2){$1$};
       		\node (destination) at (axis cs:0.65,-0.3){$0$};
       		\draw[very thick,<-,>=stealth](source)--(destination);
		\end{axis} 		
	\end{tikzpicture}}
			\caption{$\alpha=\beta \in [0,1]$.}
			\label{fig:JSDp4_pos}
		\end{subfigure}
		\caption{Jacobi SD, $p=4$, left correction functions $\alpha=\beta$ at various values.}
		\label{fig:FRJSDp4_corrs}
	\end{figure}
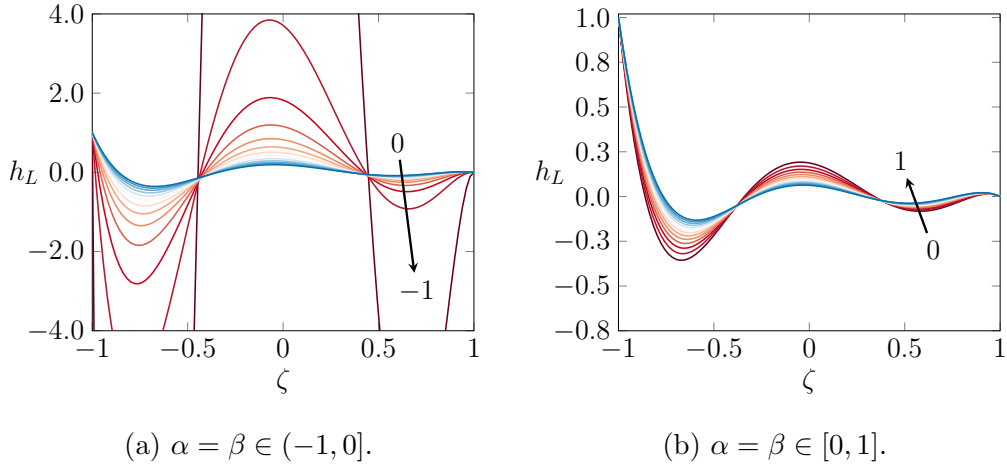

\subsection{Fully-Discrete}

	The second method used to understand these new correction functions is analysis of the explicit temporal integration stability. The importance of this is that, as is often the case in numerical methods, there is a trade off between spatial accuracy and temporal stability. We wish to know if a compromise can be found or if a scheme exists that breaks this trade off, being both accurate and temporally stable. To investigate the maximum stable time step we will use von Neumann's stability conditions~\cite{Isaacson1994} on the fully discretised form
	\begin{equation}
		\mathbf{u}^{\delta,n+1}_j = \mathbf{R}(\mathbf{Q}(k))\mathbf{u}^{\delta,n}_j,
	\end{equation}
	where $\mathbf{R}$ is the update matrix that advances the solution from the $n^{\mathrm{th}}$ time level to the $n+1^{\mathrm{th}}$ level. The definition of the update matrix varies with explicit integration method. As an example consider, forward Euler $\mathbf{R} = \mathbf{I} -\tau\mathbf{Q}$, where $\tau$ is the explicit time step and $\mathbf{I}$ is the identity matrix. Therefore, von Neumann's theorem states that for stability the spectral radius must be less than unity $\rho(\mathbf{R})\leqslant 1$ or, such that, from Banach's fixed point theorem~\cite{Brezis2010}, the update matrix must cause a contraction.

	\begin{figure}[tbhp]
		\centering
			\resizebox{0.6\linewidth}{!}{	\begin{tikzpicture}
		\begin{axis}[name=plot1,
			xlabel={$\alpha=\beta$},
			xtick={-1,-0.5,0,0.5},
			xmin=-1,xmax=0.5,
    			ylabel={$\mathrm{CFL}_\mathrm{max}$},
			ytick={0,0.05,0.1,0.15,0.2},
			ymin=0,ymax=0.2,
			y tick label style={
				/pgf/number format/.cd,
				fixed,
            			fixed zerofill,
            			precision=2,
        			/tikz/.cd
    			},
    			legend style={at={(0.025,0.97)},anchor=north west,font=\scriptsize},
    			style={font=\normalsize}]			
			
			\addplot[color={RdBu-C}, style={thick}]
				table[x=ab,y=cfl,col sep=comma,unbounded coords=jump]{./Figs/data/FRJSDp4_RK44_CFL_ab.csv};
			\addlegendentry{JSD RK44}
			\addplot[color={RdBu-F}, style={thick}]
				table[x=ab,y=cfl,col sep=comma,unbounded coords=jump]{./Figs/data/FRJSDp4_RK33_CFL_ab.csv};
			\addlegendentry{JSD RK33}
			\addplot[color={RdBu-J}, style={thick}]
				table[x=ab,y=cfl,col sep=comma,unbounded coords=jump]{./Figs/data/FRJDGp4_RK33_CFL_ab.csv};
			\addlegendentry{qDG RK33}
			\addplot[color={RdBu-M}, style={thick}]
				table[x=ab,y=cfl,col sep=comma,unbounded coords=jump]{./Figs/data/FRJDGp4_RK44_CFL_ab.csv};
			\addlegendentry{qDG RK44}
		\end{axis} 		
	\end{tikzpicture}}
		\caption{$\alpha=\beta \in (-1,0.5]$.}
		\label{fig:FRJ_CFL}
	\end{figure}
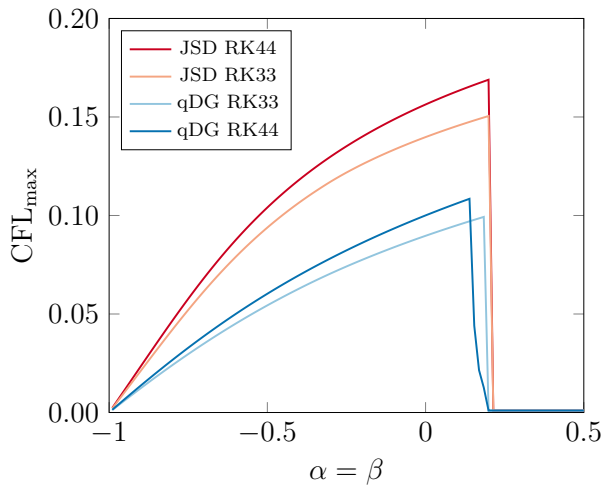

\section{Numerical Experiments}\label{sec:numeric}
	To validate the theoretical finding of the previous sections, several numerical experiments were performed.
	
	\subsection{Heterogeneous Linear Conservation}\label{sec:hlc}
	Initial investigations were for the approximation of a heterogeneous linear partial differential equation. In particular
	\begin{equation}
		\px{u}{t} + (2-\sin{x})\px{u}{x} = 0,
	\end{equation}	
	which may be cast in conservative form as
	\begin{equation}\label{eq:hlc}
		\px{u}{t} + \px{}{x}\Big[(2-\cos{x})u\Big] = u\sin{x}.
	\end{equation}
	Similar governing equations have been used to investigate the properties of FR, such as in the work of \"{O}ffner~\etal~\cite{Offner2019}. For a torus domain, $\mathbf{\Omega} \in [0,2\pi]$, this equation can be shown to be periodic in time~\cite{Trojak2018} with period $T=4\pi/\sqrt{3}$. This property allows for straightforward evaluation of the error. In particular, we investigated the error defined in the broken norm as
	\begin{equation}
		\epsilon^2 = \|u_0 - u_T\|^2_2 = \sum^{N-1}_{n=0}\int_{\mathbf{\Omega}_n}\Big[u^\delta_n(x,0) - u^\delta_n(x,T)\Big]^2\mathrm{d}x.
	\end{equation}
	The approximate value of which was calculated using the a Gauss--Legendre quadrature, facilitated by  the reference solution-points being the Gauss--Legendre nodes. The initial condition was set as 
	\begin{equation}
		u(x,0) = v + \exp{\bigg(-\frac{(x-\pi)^2}{\sigma^2}\bigg)}, \quad \mathrm{with} \quad v = 1 \quad \mathrm{and} \quad \sigma = 0.1.
	\end{equation}
	The common interface flux was set to either fully upwinded (UW) or centrally differenced (CD). Low storage RK4 explicit time stepping was used. The time step was set by dividing the time period by a fixed number of explicit steps.  
	
	\begin{figure}[tbhp]
		\centering
			\resizebox{0.6\linewidth}{!}{\begin{tikzpicture}
	\begin{axis}[
		xlabel={$\alpha=\beta$},
		xmin=-0.4,xmax=0.4,
		xtick={-0.4,-0.2,...,0.4},
		ylabel={$\epsilon$},
 		ylabel style={rotate=-90},
 		ymin=0,ymax=4e-3,
		style={font=\normalsize},		
 		legend style={draw=white!15!black,legend cell align=left,	
 		font=\footnotesize,at={(0.65,0.825)},anchor=west}]

		\addplot[color={RdBu-C}, style={very thick}]
			table[x index={1},y index={0}]{./Figs/data/bump_FRqDGp4_cd.dat};
		\addlegendentry{qDG CD}
		
		\addplot[color={RdBu-F}, style={very thick}]
			table[x index={1},y index={0}]{./Figs/data/bump_FRqDGp4_uw.dat};
		\addlegendentry{qDG UW}
		
		\addplot[color={RdBu-J}, style={very thick}]
			table[x index={1},y index={0}]{./Figs/data/bump_FRJSDp4_uw.dat};
		\addlegendentry{JSD UW}
		
		\addplot[color={RdBu-M}, style={very thick}]
			table[x index={1},y index={0}]{./Figs/data/bump_FRJSDp4_cd.dat};
		\addlegendentry{JSD CD}
	\end{axis}
\end{tikzpicture}	}
		\caption{Error comparison after one time period for $p=4$ FR, with $1200$ degrees of freedom and $\Delta t=T/10^4$.}
		\label{fig:hlc}
	\end{figure}
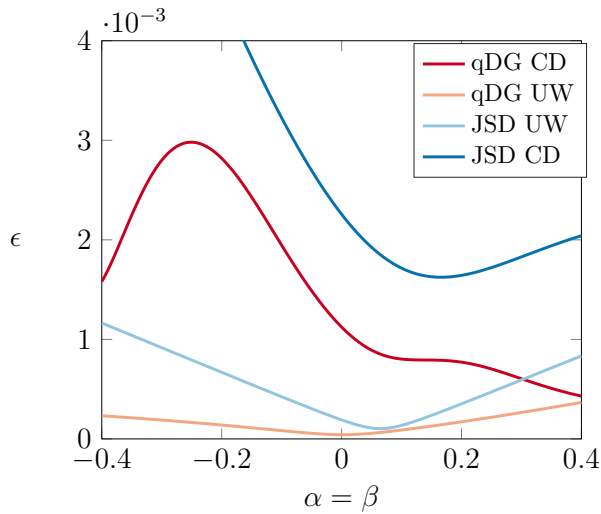
		
	A comparison of the tests performed is displayed in Fig.~\ref{fig:hlc}. The upwind results of Fig.~\ref{fig:hlc} give good agreement with the theoretical results for convergence presented in section~\ref{sec:conv}, $\epsilon$ being convex and having a clear optimum in $0<\alpha=\beta<0.2$. A characteristic of the governing equation, Eq.~(\ref{eq:hlc}), is that it introduces aliasing at the highest frequency, however a feature of unwinding---displayed in Fig.~\ref{fig:FR4JSD_dd}---is that high wavenumbers experience large dissipation. Hence, the heterogeneous equation recovers behaviour similar to that of linear advection. Centrally differenced interfaces also show low error for $0<\alpha=\beta<0.2$. Yet outside of this range the behaviour differs, this is due to the presence of aliasing and the linear system being analytically unstable. These results are indicative of the benefit of weight functions putting importance on the domain interior and the benefit of the smoother correction functions this yields.
	
	\subsection{Burgers' Turbulence}\label{sec:tb}
	Subsequently, this new set of schemes was applied to the 1D turbulent viscous Burgers equation \cite{Love1980}. The 1D viscous Burger's equation is defined as
	\begin{equation}
		\px{u}{x} + \half\px{u^2}{x} = \mu \pxi{2}{u}{x}.
	\end{equation}
	Here $\mu$ is the diffusivity and, following the investigations of San~\cite{San2016} and Alhawwary~\etal~\cite{Alhawwary2018}, it was set as $\mu=2\times10^{-4}$. The solution of this equation, due to the diffusion term, experiences a cascade of energy from the large scales to the small scales~\cite{Love1980}. This energy cascade can subsequently be derived from the closed form of the solution of Hopf~\cite{Hopf1948}. This case is of importance as it can give some indication of the behaviour of the numerical scheme when applied to Navier--Stokes turbulence, due to analogies in the mechanisms.  The energy spectra used to derive the initial velocity field was
	\begin{equation}
		E(k,t=0) = \frac{Ak^4}{k_0^5}\exp{\big(-(k/k_0)^2\big)},
	\end{equation}
with
	\begin{equation}
		A = \frac{2}{3\sqrt{\pi}} \quad \mathrm{and} \quad k_0=10,
	\end{equation}
	on a 1D periodic domain $\mathbf{\Omega} = [0,2\pi]$. This leads to $E(k,0)$ having a maximum value at $k=13$, where $k\in \mathbb{N}$. This initial condition was chosen as it is known to transition to a clear $-k^2$ energy cascade.  Therefore, the velocity field in the spatial domain may be written as:
	\begin{equation}
		u(x) = \sum_{k=0}^{k_\mathrm{max}}\sqrt{2E(k,0)}\cos{(kx+2\pi\Phi(k))} + \overline{u}
	\end{equation}
	where $2\pi\Phi(k)\in (0,2\pi]$ is a random phase angle, $k_\mathrm{max}$ is that maximum wavenumber set here to $2048$, and $\overline{u}$ is the mean velocity. For the recovery of an initial turbulence intensity of $~0.66\%$, $\overline{u}=75$~\cite{Li2016,Alhawwary2018} was set. To maintain a constant temporal filter width with RK44 integration, the initial CFL number ($\overline{u}\Delta t/\Delta x$) was held at $0.057$. This CFL number is equivalent to $\Delta t=2\times10^{-5}$ for $n=1200$, $p=4$, as was used by Alhawwary~\etal~\cite{Alhawwary2018}.

	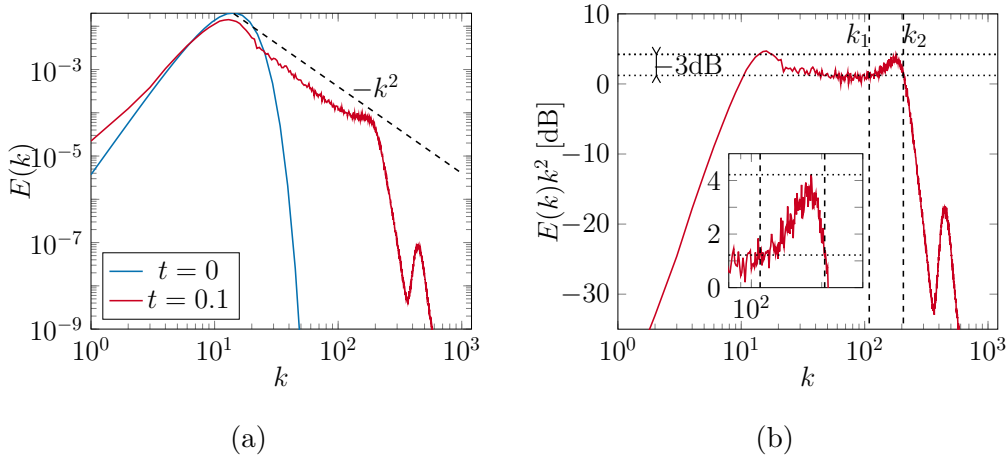
\begin{figure}[tbhp]
		\centering
		\begin{subfigure}[b]{0.48\linewidth}
			\centering
				\resizebox{\linewidth}{!}{\begin{tikzpicture}
	\begin{axis}[name=plot1,ylabel={$E(k)$},xlabel={$k$},
	    xtick={1,10,100,1000},ytick={1e-9,1e-8,1e-7,1e-6,1e-5,1e-4,1e-3,1e-2},
	    yticklabels={$10^{-9}$, ,$10^{-7}$, ,$10^{-5}$, ,$10^{-3}$},
   		xmode=log,
   		xmin=1,xmax=1200,
   		ymode=log,
   		ymin=1e-9,ymax=2e-2,
   		style={font=\large},
   		legend pos=south west]

		\addplot[color={RdBu-M}, style={thick}]
			table[x=k,y=E,col sep=comma,unbounded coords=jump]{./Figs/data/E0.csv};
		\addplot[color={RdBu-C}, style={thick}]
			table[x=f,y=P,col sep=comma,unbounded coords=jump]{./Figs/data/bf240_sd50.csv};
		\addplot[color=black,dashed,style={thick}] coordinates { (10,1e-2*4) (1000,1e-6*4)};
   		\node[] at (axis cs: 200,4e-4) {$-k^2$};
   		\legend{$t=0$,$t=0.1$};
	\end{axis}		
\end{tikzpicture}}
			\caption{}
			\label{fig:sd_spec1}
		\end{subfigure}
		~
		\begin{subfigure}[b]{0.48\linewidth}
			\centering
				\resizebox{\linewidth}{!}{\begin{tikzpicture}
	\begin{axis}[name=plot1,ylabel={$E(k)k^2 \; [\mathrm{dB}]$},xlabel={$k$},
	    xtick={1,10,100,1000},ytick={-40,-30,-20,-10,0,10},
   		xmin=1,xmax=1200,
   		xmode=log,
   		ymin=-35,ymax=10,
   		style={font=\large},
   		after end axis/.code={
            \draw[black,>=angle 60,>-<,semithick] (axis cs:2.05,5.3166) -- (axis cs:2.05,0.1166);
             }]]
             
			\addplot[color=black,dotted,style={thick}] coordinates { (1,4.2166) (1200,4.2166)};
			\addplot[color={RdBu-C}, style={thick}]
				table[x=f,y=Pdb,col sep=comma,unbounded coords=jump]{./Figs/data/bf240_sd50.csv};
			\addplot[color=black,dotted,style={thick}] coordinates { (1,4.2166) (1200,4.2166)};
			\addplot[color=black,dotted,style={thick}] coordinates { (1,1.2166) (1200,1.2166)};
			\addplot[color=black,dashed,style={thick}] coordinates { (206,10) (206,-35)};
			\addplot[color=black,dashed,style={thick}] coordinates { (109,10) (109,-35)};
    		\node[] at (axis cs: 89,7) {$k_1$};
    		\node[] at (axis cs: 256,7) {$k_2$};
    		\node[] at (axis cs: 3.7,2.6) {$-3\mathrm{dB}$};
	\end{axis}
		
	\begin{axis}[name=plot1,
		xtick={1,10,100,1000},ytick={0,1,2,3,4},
		yticklabels={0, ,2, ,4},
    	xmin=80,xmax=300,
    	xmode=log,
    	ymin=0,ymax=5,
    	style={font=\large},
    	height=4cm,width=4cm,
    	shift={(2cm,0.75cm)},axis background/.style={fill=white}]
		
		\addplot[color={RdBu-C}, style={thick}]
			table[x=f,y=Pdb,col sep=comma,unbounded coords=jump]{./Figs/data/bf240_sd50.csv};
		\addplot[color=black,dotted,style={thick}] coordinates { (80,4.2166) (300,4.2166)};
		\addplot[color=black,dotted,style={thick}] coordinates { (80,1.2166) (300,1.2166)};
		\addplot[color=black,dashed,style={thick}] coordinates { (206,5) (206,0)};
		\addplot[color=black,dashed,style={thick}] coordinates { (109,5) (109,0)};
	\end{axis}
\end{tikzpicture} }
			\caption{}
			\label{fig:sd_spec2}
		\end{subfigure}
		\caption{Burgers' turbulence energy spectra of FR, $p=4$, with SD correction functions $\alpha=\beta=0$, DoF = 1200, $\Delta t=2\times10^{-5}$.}
		\label{fig:sd_tb_spectra}
	\end{figure}

	The primary means of evaluating schemes for this case is through the energy spectra and here the spectra was averaged over one hundred runs. Studying Fig.~\ref{fig:sd_spec1}, it can be seen that after $0.1s$ the expected energy cascade has become established. A feature that is commonly seen when numerical solving this equation with ILES is the resonant peak---in this case at $k\approx100$. This peak has previously been explained~\cite{San2016} as a pile-up of energy at the smallest scales due to under-dissipation, however it is more readily explained through bottlenecking and hyper-viscosity~\cite{Frisch2013}. Regardless of the precise mechanism, it is advantageous to reduce the size of this resonant peak when explicit filtering is not used--such as in ILES.

If instead the compensated energy spectra~\cite{Frisch2013} is studied, a clear peak can be seen which is analogous to those seen in mechanical vibrations and control theory. The peak observed was then parametrised using Q-factor~\cite{Green1955} defined as
	\begin{equation}
		Q = \frac{k_0}{k_2-k_1}.
	\end{equation}
	Here $k_0$ is the wavenumber of the resonance peak, and $k_1$ and $k_2$ are the half power wavenumbers i.e. the wavenumbers for the peak $-3\mathrm{dB}$. An example of the modified energy spectra, together with the locations of $k_1$ and $k_2$, is shown in Fig.~\ref{fig:sd_spec2}. The Q-factor of this example is $1.845$. Q-factor is an indicator of the dissipation that is occurring with in the system, with higher Q-factors implying less dissipation and vice versa. A second parameter we will examine is the cut-off wavenumber, which is defined as the wavenumber at which the value of $E(k)k^2$ drops by $3\mathrm{dB}$ from the level of the plateau.

	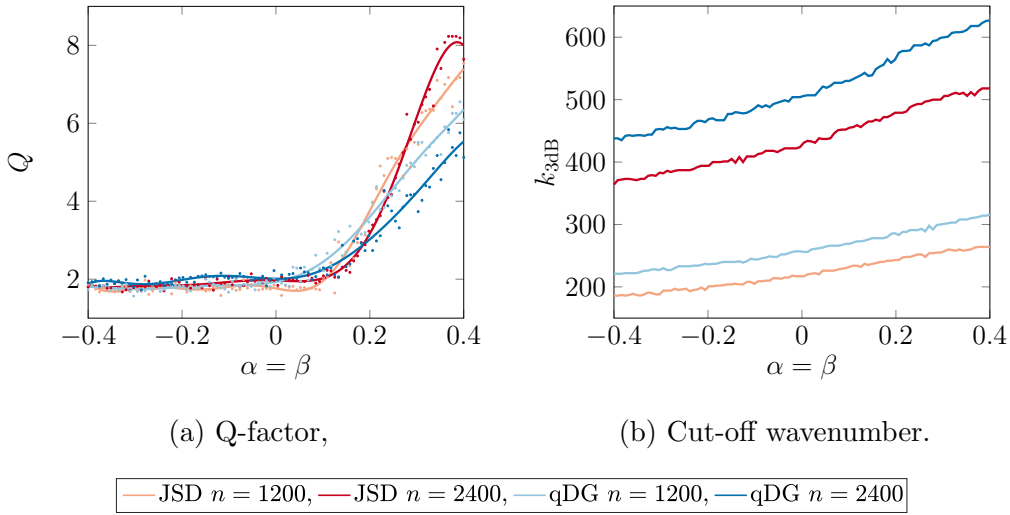
\begin{figure}[tbhp]
		\centering
		\begin{subfigure}[b]{0.48\linewidth}
			\centering
				\resizebox{\linewidth}{!}{\begin{tikzpicture}
	\begin{axis}[name=plot1,ylabel={$Q$},xlabel={$\alpha=\beta$},
		xtick={-0.4,-0.2,...,0.4},ytick={0,2,...,10},
    	xmin=-0.4,xmax=0.4,
    	ymin=1,ymax=9,
    	style={font=\large},
    	legend pos=north west]
		
		\addplot[only marks,color={RdBu-F},mark=*,mark size=0.5pt]
			table[x=c,y=Q,col sep=comma,unbounded coords=jump]{./Figs/data/sd240.csv};				
		\addplot[only marks,color={RdBu-C},mark=*,mark size=0.5pt]
			table[x=c,y=Q,col sep=comma,unbounded coords=jump]{./Figs/data/sd480.csv};				
		\addplot[only marks,color={RdBu-J},mark=*,mark size=0.5pt]
			table[x=c,y=Q,col sep=comma,unbounded coords=jump]{./Figs/data/dg240.csv};				
		\addplot[only marks,color={RdBu-M},mark=*,mark size=0.5pt]
			table[x=c,y=Q,col sep=comma,unbounded coords=jump]{./Figs/data/dg480.csv};
			
		\addplot[color={RdBu-F}, style={very thick}]
			table[x=x,y=q3,col sep=comma,unbounded coords=jump]{./Figs/data/Q_fit.csv};
		\addplot[color={RdBu-C}, style={very thick}]
			table[x=x,y=q4,col sep=comma,unbounded coords=jump]{./Figs/data/Q_fit.csv};
		\addplot[color={RdBu-J}, style={very thick}]			
			table[x=x,y=q1,col sep=comma,unbounded coords=jump]{./Figs/data/Q_fit.csv};
		\addplot[color={RdBu-M}, style={very thick}]
			table[x=x,y=q2,col sep=comma,unbounded coords=jump]{./Figs/data/Q_fit.csv};
	\end{axis}		
\end{tikzpicture}}			
			\caption{\label{fig:tb_q}Q-factor,}
		\end{subfigure}
		~
		\begin{subfigure}[b]{0.48\linewidth}
			\centering
				\resizebox{\linewidth}{!}{\begin{tikzpicture}
	\begin{axis}[name=plot1,ylabel={$k_{3\mathrm{dB}}$},xlabel={$\alpha=\beta$},
		xtick={-0.4,-0.2,...,0.4},ytick={200,300,...,600},
    	xmin=-0.4,xmax=0.4,
    	ymin=150,ymax=650,
    	style={font=\large},
    	legend pos=north west]
		
		\addplot[color={RdBu-F}, style={very thick}]
			table[x=c,y=f3,col sep=comma,unbounded coords=jump]{./Figs/data/sd240.csv};				
		\addplot[color={RdBu-C}, style={very thick}]
			table[x=c,y=f3,col sep=comma,unbounded coords=jump]{./Figs/data/sd480.csv};				
		\addplot[color={RdBu-J}, style={very thick}]
			table[x=c,y=f3,col sep=comma,unbounded coords=jump]{./Figs/data/dg240.csv};				
		\addplot[color={RdBu-M}, style={very thick}]
			table[x=c,y=f3,col sep=comma,unbounded coords=jump]{./Figs/data/dg480.csv};
	\end{axis}		
\end{tikzpicture}}
			\caption{\label{fig:tb_cutoff}Cut-off wavenumber.}
		\end{subfigure}
		\vspace{1em}
		\begin{subfigure}[b]{0.77\linewidth}
			\centering
			\includegraphics[width=\linewidth]{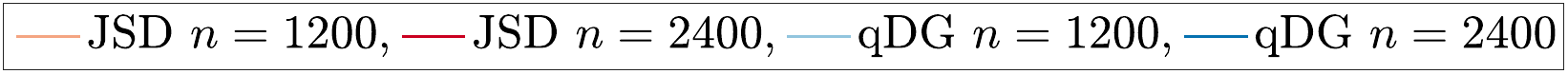}
		\end{subfigure}
		\caption{Burgers' turbulence parametrisation for FR, $p=4$, at $t=0.1$ for various correction functions.}
		\label{fig:GJ_param}
	\end{figure}

	Figure~\ref{fig:GJ_param} displays the Q-factor and cut-off wavenumber for qDG and SD, exploring the behaviour of $\alpha=\beta$ for two different degree of freedom counts;  $n=1200$ and $n=2400$. Studying Figure~\ref{fig:tb_q} there is a region $-0.4<(\alpha=\beta)<\approx 0.1$ where the Q-factor is largely invariant with the correction function. The increase in Q-factor for $\alpha=\beta>\approx0.1$ is consistent with the semi-discrete stability presented in Figure~\ref{fig:FR4_up_stab}, and although the schemes were stable here the large increase in Q-factor is undesirable.
	
	The findings of the cut-off wavenumber are displayed in Fig.~\ref{fig:tb_cutoff}. These show that qDG has a higher cut-off wavenumber than SD and that, within the range tested, the cut-off frequency rose as $\alpha=\beta$ was increased.  These two findings suggest that a good correction function is in the vicinity of $0<\alpha=\beta<0.2$ as these are both analytically stable and provide an increased cut-off wavenumber. This is consistent with the hypothesis that weighting the interior of the reference domain is beneficial and can be further understood through studying the dispersion and dissipation relations in Fig.~\ref{fig:FR4JSD_dd}.

\section{Conclusions}
\label{sec:conclusions}

	A new set of FR correction functions, defined in terms of Jacobi polynomials with stability affirmed through Fourier analysis, was defined using a weighted norm. This family proves the feasibility of incorporating a weight function into the definition of an FR scheme.  Using these functions it is possible to recover, at least for a linear flux function, a range of spectral difference schemes.  Theoretical Fourier analysis studies were then used to show that these new schemes can be used to increase wave resolution, while convergence studies showed that these new schemes can maintain high order convergence with increased temporal stability. The norm that is used to define this set can also permit the definition of a quasi-DG set, defined in a weighted $\ell_2$ norm, but ultimately found to not be the same as DG.  Numerical experiments were then performed to validate the theoretical findings, which showed that the optimal correction functions for a heterogeneous linear PDE and Burgers' turbulence lay approximately in the same region predicted by theoretical investigations. The numerical results are indicative that this correction function family may be of use when compared to existing schemes applied to ILES. A summary of how the present work is positioned relative to the literature is shown in Fig.~\ref{fig:venn_diagram2}. Furthermore, there are some areas addressed by this paper which may be furthered by studying how to pick a suitable correction function, and is it feasible to locally or globally adapt the correction function based on the solution.
	
	\begin{figure}
		\centering
			\resizebox{0.7\linewidth}{!}{\begin{tikzpicture}[every node/.style={font=\small},scale=2.0]
	\def\esfr{(0,0) ellipse (0.75 and 1)}
	\def\osfr{(-0.25,0) ellipse (0.5 and 0.5)}
	\def\gjfr{(-0.5,0) ellipse (1.5 and 0.52)}
	\def\gsfr{(0.75,0) ellipse (1.5 and 1.5)}
	\definecolor{col1}{RGB}{202,0,32}
	\definecolor{col2}{RGB}{244,165,130}
	\definecolor{col3}{RGB}{146,197,222}
	\definecolor{col4}{RGB}{25,113,176}
   
    \begin{scope}[]
        
        \draw \esfr node[label={[xshift=0em, yshift=2em]\small{ESFR}}]{};
        \draw \osfr node[label={[xshift=0em, yshift=-1em]\small{OSFR}}]{};
        \draw \gjfr node[label={[xshift=-3em, yshift=-1em]\small{GJFR}}]{};
        \draw \gsfr node[label={[xshift=4em, yshift=-1em]\small{GSFR}}]{};
        
        \fill[fill=col1] (1.1,0.95) rectangle (1.2,1.05);
        \fill[fill=col2,rotate around={45:(0.25,0)}] (0.2,-0.05) rectangle (0.3,0.05);
        \fill[fill=col3] (-0.5,-0.35) rectangle (-0.4,-0.25);
        \fill[fill=col4] (-0.75,0) circle (0.05);
        
        \fill[fill=col1] (2.4,0.45) rectangle (2.5,0.55);
        \node[anchor=west] at (2.55,0.5) {Huynh LCL};
        
        \fill[fill=col2,rotate around={45:(2.45,0.2)}] (2.4,0.15) rectangle (2.5,0.25);
        \node[anchor=west] at (2.55,0.2) {LSD};
        
        \fill[fill=col3] (2.4,-0.15) rectangle (2.5,-0.05);
        \node[anchor=west] at (2.55,-0.1) {Huynh $g_2$};
        
        \fill[fill=col4] (2.45,-0.4) circle (0.05);
        \node[anchor=west] at (2.55,-0.4) {NDG};
    \end{scope}

\end{tikzpicture}}
		\caption{Euler diagram to show the interconnection of the spaces of FR correction functions: Nodal DG~(NDG)~\cite{Huynh2007}; Original Stable FR~(OSFR)~\cite{Vincent2010}; Extended range Stable FR~(ESFR)~\cite{Vincent2015}; Generalised Sobolev stable FR~(GSFR)~\cite{Trojak2018} and, the Generalised Jacobi stable FR~(GJFR) of the present work. Some specific examples of specific schemes are given, notably Huynh's Lumped Chebyshev-Lobatto~(LCL)~\cite{Huynh2007} scheme and the original Legendre spectral difference (LSD) scheme~\cite{Kopriva1996,Huynh2007,Jameson2010}.}
		\label{fig:venn_diagram2}
	\end{figure}
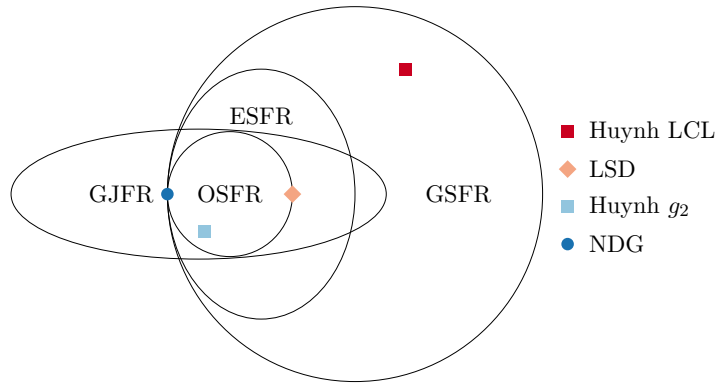

\section*{Acknowledgements}
\label{sec:ack}
	The authors would like to acknowledge Rob Watson for his comments, suggestions, and contributions The support of the Engineering and Physical Sciences Research Council of the United Kingdom is gratefully acknowledged under the award reference 1750012. The authors would also like to thank Nvidia for the GPU Seeding Grant received.

\section*{References}
\bibliographystyle{elsarticle-num}
\bibliography{library}


\clearpage
\begin{appendices}
\section{Nomenclature}
		\begin{tabbing}
		  XXXXXXXX \= \kill
		  	\textit{Roman}\\
		  	$A_p$ \> constant relating the ratio of $\tilde{h}_{L,p-1}$ over $\tilde{h}_{L,p+1}$ and $\iota$\\
			$b_p^{(\alpha,\beta)}$ \> $p^{\mathrm{th}}$ derivative of $p^{\mathrm{th}}$ Jacobi polynomial\\		  	
			$c$ \> OSFR correction function variable ($\iota = c/2$)\\
		  	$D_{i,j}(\gamma,\delta,\alpha,\beta)$ \> Doha's Jacobi differentiation constant\\
		  	$\mathbf{e}_j(t,J)$ \> analytical error at solution points for time $t$ with element Jacobian $J$\\
		  	$e_1(t)$ \> error in $\epsilon_1$ at time t\\
		  	$e_2(t)$ \> error in $\epsilon_2$ at time $t$\\
		  	$h_L$~\&~$h_R$ \> left and right correction functions\\
		  	$J_n$ \> $n^{\mathrm{th}}$ Jacobian of the mapping $x\rightarrow\zeta$ \\
		  	$J_i^{(\alpha,\beta)}$ \> $i^{\mathrm{th}}$ order Jacobi polynomial of the first kind\\
		  	$k$ \> wavenumber\\
		  	$L_n$ \> $n^{\mathrm{th}}$ order Legendre polynomial\\
			$q_n^{(\alpha,\beta)}$ \> value of the $n^{\mathrm{th}}$ Jacobi orthogonality integral\\
			$\mathbf{Q}$ \> FR operator matrix for linear advection equation\\
			$r_h(t)$ \> rate of convergence with cell width at time $t$\\
			$w_{\alpha,\beta}(x)$ \> Jacobi weight function $(1-x)^{\alpha}(1+x)^{\beta}$\\
			$\mathbf{w}_n$ \> $n^{\mathrm{th}}$ column of $\mathbf{W}$\\
			$\mathbf{W}$ \> eignevector matrix of $\mathbf{Q}$\\
			$W_{2}^{c}$ \> Sobolev space in 2 norm, modified by factor $c$\\
			$W_{2}^{\iota,w}$ \> Sobolev space in 2 norm, modified by factor $\iota$ and weighted by $w_{\alpha,\beta}$\\
				
			\\ \textit{Greek}\\
			$\alpha$ \> first Jacobi polynomial control parameter\\
			$\beta$ \> second Jacobi polynomial control parameter\\
			$\gamma$ \> ratio of $\Gamma$ functions\\
			$\Gamma(x)$ \> Gamma function evaluated at $x$\\
			$\zeta$ \> 1D spatial variable in reference domain\\
			$\eta_p$ \> OSFR derived parameter\\
			$\Theta_n$ \> linear transformation of $n^{\mathrm{th}}$ element from $x$ to $\zeta$. $\Theta_n: \mathbf{\Omega}_n\rightarrow\hat{\mathbf{\Omega}}$\\
			$\iota$ \> correction function parameter\\
			$\iota_{\mathrm{crit}}$ \> critical correction function parameter\\
			$\kappa_p^{(\alpha,\beta)}$ \> derived parameter for Jacobi correction functions\\
			$\mathbf{\Lambda}$ \> diagonal eignevalue matrix of $\mathbf{Q}$\\
			$\mu$ \> dynamic viscosity\\
			$\mathbf{\Omega}$ \> spatial domain\\
			$\mathbf{\Omega}_n$ \> $n^{\mathrm{th}}$ partition of the domain $\mathbf{\Omega}$\\

			\\ \textit{Subscript}\\
			$\mathrm{\bullet}_L$ \> variable at left of cell\\
			$\mathrm{\bullet}_R$ \> variable at right of cell\\
			
			\\ \textit{Superscript}\\
			$\mathrm{\bullet}^{\delta C}$ \> correction to the discontinuous function\\
			$\mathrm{\bullet}^{\delta D}$ \> uncorrected discontinuous value\\
			$\mathrm{\bullet}^{\delta I}$ \> common value at interface\\
			$\mathrm{\bullet}^T$ \> vector or matrix transpose\\
			$\mathrm{\bullet}^{\delta}$ \> discontinuous value\\
			$\hat{\mathrm{\bullet}}$ \> variable transformed to reference domain\\
			$\overline{\mathrm{\bullet}}$ \> function averaged of the reference domain\\
			$\tilde{\mathrm{\bullet}}$ \> function transformed to polynomial space\\
			
			\\ \textit{Symbols or Operators} \\
			$\poch{x}{i}$ \> rising Pochhammer function of $x$ with $i$ steps\\
			$_nF_m(\mathbf{N},\mathbf{M};z)$ \> the n-m generalised hypergeometric function at $z$ \cite{Bailey1933,Bateman1953}\\

			\\ \textit{Abbreviations}\\
			CPR \> Correction Procedure via Reconstruction\\
			DG \> Discontinuous Galerkin\\
			FR \> Flux Reconstruction\\
			GJFR \> Generalised Jacobi stable Flux Reconstruction\\
			GLSFR \> Generalised Lebesgue Stable Flux Reconstruction\\
			GSFR \> Generalised Sobolev stable Flux Reconstruction\\
			ILES \> Implicit Large Eddy Simulation\\
			LCL \> Lumped Chebyshev-Lobatto\\
			LCP \> Lifting Collocation Penalty\\ 
			LES \> Large Eddy Simulation\\
			LSD \> Legendre Spectral Difference\\
			NDG \> Nodal Discontinuous Galerkin\\
			OSFR \> Original Stable Flux Reconstruction\\
			qDG \> quasi Discontinuous Galerkin\\
			SD \> Spectral Difference\\
			TGV \> Taylor-Green Vortex\\
			
		\end{tabbing}
\end{appendices}


\end{document}